\documentclass{siamart0516}
\usepackage{amsfonts}
\usepackage{amsmath}
\usepackage{amssymb}
\usepackage{array}
\usepackage{graphicx}
\usepackage{stmaryrd}
\usepackage{marginnote}
\usepackage{color}
\usepackage{xspace}

\newcommand{\Ek}{\mathbf{P}_{h}}

\newcommand{\ak}{\tilde{a}_h}
\newcommand{\mk}{\tilde{m}_h}

\newcommand{\Vh}{\mathbb{V}_{h}}

\newcommand{\N}{\mathbf{N}}

\newcommand{\W}{\mathbb{W}}

\newcommand{\bnu}{\ensuremath{\boldsymbol \nu}}
\newcommand{\bpsi}{\ensuremath{\boldsymbol \psi}}
\newcommand{\bPsi}{\ensuremath{\boldsymbol \Psi}}

\newcommand{\bxi}{\ensuremath{\boldsymbol \xi}}
\newcommand{\bF}{\ensuremath{\boldsymbol F}}
\newcommand{\bN}{\ensuremath{\boldsymbol N}}
\newcommand{\bl}{\ensuremath{\boldsymbol l}}
\newcommand{\bL}{\ensuremath{\boldsymbol L}}
\newcommand{\bG}{\ensuremath{\boldsymbol G}}
\newcommand{\bP}{\ensuremath{\boldsymbol P}}
\newcommand{\bZ}{\ensuremath{\boldsymbol Z}}

\newtheorem{remark}[theorem]{Remark}

\newcommand{\step}[1]{\noindent\raisebox{1.5pt}[10pt][0pt]{\tiny\framebox{$#1$}}\xspace}

\numberwithin{equation}{section}
\numberwithin{theorem}{section}
\numberwithin{figure}{section}

\title{A priori error estimates for finite element approximations to eigenvalues and eigenfunctions of the Laplace-Beltrami operator}
\author{
  Andrea Bonito\thanks{Department of Mathematics, Texas A\&M University, College Station TX, 77843; email: {\tt bonito@math.tamu.edu}. 
  Partially supported by NSF Grant DMS-1254618.
}
\and
Alan Demlow\thanks{Department of Mathematics, Texas A\&M University, College Station TX, 77843; email: {\tt demlow@math.tamu.edu}.
Partially supported by NSF Grants DMS-1518925 and DMS-1720369.
}
\and
Justin Owen\thanks{Department of Mathematics, Texas A\&M University, College Station, TX, 77843; email: {\tt jowen6@math.tamu.edu}.  Partially supported by NSF Grants DMS-1518925 and DMS-1720369.
}
}

\begin{document}

% Short title for running heads:
%\shorttitle{FEM approximation of Laplace-Beltrami eigenproblems}

\maketitle

\begin{abstract}
Elliptic partial differential equations on surfaces play an essential role in geometry, relativity theory, phase transitions, materials science, image processing, and other applications. They are typically governed by the Laplace-Beltrami operator. We present and analyze approximations by Surface Finite Element Methods (SFEM) of the Laplace-Beltrami eigenvalue problem.  As for SFEM for source problems, spectral approximation is challenged by two sources of errors: the geometric consistency error due to the approximation of the surface and the Galerkin error corresponding to finite element resolution of eigenfunctions. We show that these two error sources interact for eigenfunction approximations as for the source problem.  The situation is different for eigenvalues,  %When there is no consistency error, the error decays as the square of the energy error in the corresponding eigenfunctions.  
where a novel situation occurs for the geometric consistency error: The degree of the geometric error depends on the choice of interpolation points used to construct the approximate surface.   Thus the geometric consistency term can sometimes be made to converge faster than in the eigenfunction case through a judicious choice of interpolation points.    
%For SFEM the eigenvalue error instead decays as the square of the best approximation to the eigenfunction in the energy norm, plus a consistency term that depends on the approximation of the geometry. Here a novel situation occurs: The degree of the geometric error for the eigenvalue approximations depends on the choice of interpolation points used to construct the approximation of the surface. We show that the geometric consistency term can sometimes be made to converge faster than in the eigenfunction case through a judicious choice of interpolation points.    
\end{abstract}


\begin{AM} 65N12, 65N15, 65N25, 65N30
\end{AM}

\pagestyle{myheadings}
\thispagestyle{plain}

\markboth{A. BONITO, A. DEMLOW, AND J. OWEN}{FEM FOR LAPLACE-BELTRAMI EIGENVALUE PROBLEMS}

\section{Introduction}

The spectrum of the Laplacian is ubiquitous in the sciences and engineering. Consider the eigenvalue problem $-\Delta u= \lambda u$ on a Euclidean domain $\Omega$, with $u=0$ on $\partial \Omega$.  There is then a  sequence $0 < \lambda_1  \le \lambda_2 \le \lambda _3 \le ...$ of eigenvalues with corresponding $L_2$-orthonormal eigenfunctions $\{ u_i\}$.  Given a finite element space $\mathbb{V} \subset H_0^1(\Omega)$, the natural finite element counterpart is to find $(U, \Lambda) \in \mathbb{V} \times \mathbb{R}^+$ such that $\int_\Omega \nabla U \cdot \nabla V = \Lambda \int_\Omega UV$, $V \in \mathbb{V}$.   

Finite element methods (FEM) are a natural and widely used tool for approximating spectra of elliptic PDE.  
Analyzing the error behavior of such FEM is more challenging than for source problems because of the nonlinear nature of the problem.  A priori error estimation for FEM approximations of the eigenvalues and eigenfunctions of the Laplacian and related operators in flat (Euclidean) space is a classical topic in finite element theory; cf. \cite{PV72, Ch74, BO87, BO89}.  We highlight the review article \cite{BO91} of Babu\v{s}ka and Osborn in this regard.  These bounds are all asymptotic in the sense that they require an initial fineness condition on the mesh.  More recently, sharp bounds for eigenvalues (but not eigenfunctions) appeared in \cite{KnOs06}.  These bounds are notable because they are truly a priori in the sense that they do not require a sufficiently fine mesh.  Finally, over the past decade a number of papers have appeared analyzing convergence and optimality of adaptive finite element methods (AFEM) for eigenvalue problems \cite{DXZ08, GG09, CaGe11, DHZ15, Gallistl, DB16}.  Because sharp a priori estimates are needed in order to analyze AFEM optimality properties, some of these papers also contain improved a priori estimates.  We particularly highlight \cite{CaGe11, Gallistl} as our analysis of eigenfunction errors below largely employs the framework of these papers. 

Assume that a simple eigenpair $(\lambda,u)$ of $-\Delta$ is approximated using a degree-$r$ finite element space in the standard way.  
Roughly speaking, it is known that
\begin{align}
\label{flat_results_eigenfunction}
\|u-{\bf Z} u\|_{H^1(\Omega)} & \leq C \|u-{\bf G} u\|_{H^1(\Omega)} \le C h^r|u|_{H^{r+1}}, 
\\ |\lambda-\Lambda| & \le C(\lambda) \|u-{\bf G} u\|_{H^1(\Omega)}^2 \le C(\lambda) h^{2r}.   
\label{flat_results_eigenvalue}
\end{align}
Here $\Lambda$ is the discrete eigenvalue corresponding to $\lambda$, ${\bf G}$ is the Ritz projection, and ${\bf Z}$ is the Galerkin (energy) projection onto the discrete invariant space corresponding to $\Lambda$.  \eqref{flat_results_eigenfunction} holds for $h$ sufficiently small \cite{CaGe11, Gallistl}, while \eqref{flat_results_eigenvalue} holds assuming certain algebraic conditions on the spectrum \cite{KO06}.  Also, the constants in the first estimate are asymptotically independent of $\lambda$, while the constants in the second estimate depend in essence on the separation of $\lambda$ from the remainder of the spectrum and the degree to which the discrete spectrum respects that separation.  Corresponding ``cluster-robust'' estimates also hold for simultaneous approximation of clusters of eigenvalues.

We next describe surface finite element methods (SFEM).  Let $\gamma \subset \mathbb{R}^{D+1}$ be a smooth, closed, orientable $D$-dimensional surface, and let $\Delta_\gamma$ be the Laplace-Beltrami operator on $\gamma$.  The SFEM corresponding to the cotangent method was introduced by Dziuk \cite{Dziuk} in 1988.  Let $\Gamma$ be a polyhedral approximation to $\gamma$ having triangular faces which also serve as the finite element mesh.  The finite element space $\mathbb{V}$ consists of functions which are piecewise linear over $\gamma$, and we seek $U \in \mathbb{V}$ such that $\int_{\Gamma} \nabla_{\Gamma} U \cdot\nabla_{\Gamma} V = \int_{\Gamma} f V$, $V \in \mathbb{V}$.  In \cite{D09} Demlow developed a natural higher order analogue to this method.  SFEM exhibit two error sources, a standard Galerkin error and a geometric consistency error due to the approximation of $\gamma$ by $\Gamma$. Let $\mathbb{V}$ be a Lagrange finite element space of degree $r$ over a degree-$k$ polynomial approximation $\Gamma$, and let ${\bf G}$ be the Ritz projection onto $\mathbb{V}$.  Then (cf. \cite{Dziuk, D09}) 
\begin{align}
\|u - {\bf G} u\|_{H^1(\gamma)} &\leq C( h^r + h^{k+1}),
\label{D-abound}
\\
\|u - {\bf G} u - \left(\int_{\gamma}u - {\bf G} u d\sigma\right)\|_{L_2(\gamma)} &\leq C( h^{r+1} + h^{k+1}).
\label{D-mbound}
\end{align} 

The need for accurate approximations to Laplace-Beltrami eigenpairs arises in a variety of applications.  One approach to shape classification is based on the Laplace-Beltrami operator's spectral properties \cite{RWP05, RWP06, ReWoShNi09, RBGPS09, R10, KLO16, RSDCT09}.  For example, the spectrum has been used as a ``shape DNA'' to yield a fingerprint of a surface's shape.  One prototypical application is medical imaging.  There the underlying surface $\gamma$ is not known precisely, but is instead sampled via a medical scan.  The spectrum that is studied is thus that of a reconstructed approximate surface, often as a polyhedral approximation (triangulation).  Bootstrap methods are another potential application of Laplace-Beltrami spectral calculations \cite{BC15}.  Finally, Laplace-Beltrami eigenvalues on subsurfaces of the sphere characterize singularities in solutions to elliptic PDE arising at vertices of polyhedral domains \cite{Da88, KMR01, MR10}.    Many of these papers use surface FEM in order to calculate Laplace-Beltrami spectral properties.  While these methods show empirical evidence of success, there has to date been no detailed analysis of the accuracy of the eigenpairs calculated using SFEM.  Some of these papers also propose using higher-order finite element methods to improve accuracy, but do not suggest how to properly balance discretization of $\gamma$ with the degree of the finite element space.  A main goal of this paper is to provide clear guidance about the interaction between geometric consistency and Galerkin errors in the context of spectral problems.  

In this paper we develop error estimates for the SFEM approximation of the eigenpairs of the Laplace-Beltrami operator.  In particular, we develop a priori error estimates for the SFEM approximations to the solution of
$$
-\Delta_\gamma u = \lambda u \text{ on } \gamma.
$$
Let $0=\lambda_0 < \lambda_1 \le \lambda_2 \le ...$ be the Laplace-Beltrami eigenvalues with corresponding $L_2(\gamma)$-orthonormal eigenfunctions $\{ u_i\}$.   We show that the eigenvector error converges as the error for the source problem, up to a geometric term.   Our first main result is:
\begin{equation}
\label{surf_est}
\|u_i-{\bf Z} u_i\|_{H^1(\gamma)} \le C \|u_i-{\bf G} u_i\|_{H^1(\gamma)} +C(\lambda_i) h^{k+1} \le C(\lambda_i)( h^r + h^{k+1}).
\end{equation}
We also prove $L_2$ error bounds and explicit upper bound for $C(\lambda_i)$ in terms of spectral properties.  In addition to eigenfunction convergence rates, we prove the cluster robust estimate for the eigenvalue error:
\begin{equation}
|\lambda_i-\Lambda_i|  \le C(\lambda_i) (\|u_i-{\bf G} u_i\|_{H^1(\gamma)}^2 +h^{k+1}) \le C(\lambda_i)( h^{2r}+h^{k+1}),
\label{ClusterEigBound}
\end{equation}
where as above, explicit bounds for $C(\lambda_i)$ are given below.

Numerical results presented in Section \ref{sec7} reveal that \eqref{ClusterEigBound} is not sharp for $k>1$. The deal.ii library \cite{BHK:07} uses quadrilateral elements and Gauss-Lobatto points to interpolate the surface. The geometric consistency error for every shape we tested using deal.ii was found to be $O(h^{2k})$ rather than $O(h^{k+1})$ as in \eqref{ClusterEigBound}. This inspired our second main result which is stated in Theorem \ref{Superconvergence} in Section \ref{sec6}:
\begin{equation*}
|\lambda_i-\Lambda_i|\lesssim h^{2r} + h^{2k} + h^\ell.
\end{equation*}
Here $\ell$ is the order of the quadrature rule associated with the interpolation points used to construct the surface. Thus with judicious choice of interpolation points, it is possible to obtain superconvergence for the geometric consistency error when $k>1$. This phenomenon is novel as a geometric error of order $h^{k+1}$ has been consistently observed in the literature for a variety of error notions. We also investigate this framework in the context of one-dimensional problems and triangular elements. 

We finally comment on our proofs.  Geometric consistency errors fit into the framework of variational crimes \cite{SF73}.  Banerjee and Osborn \cite{BO90, Ban92} considered the effects of numerical integration on errors in finite element eigenvalue approximations, but did not provide a general variational crimes framework. Holst and Stern analyzed variational crimes analysis for surface FEM within the finite element exterior calculus framework and also briefly consider eigenvalue problems \cite{HoSt2012}.  Their discussion of eigenvalue problems does not include convergence rates or a detailed description of the interaction of geometric and Galerkin errors. The recent paper \cite{CEM14} gives a variational crimes analysis for eigenvalue problems that applies to surface FEM.  However, their analysis yields suboptimal convergence of the geometric errors in the eigenvalue analysis, considers a different error quantity than we do, and does not easily allow for determination of the dependence of constants in the estimates on spectral properties.  

In Section \ref{sec2} we give preliminaries. In Section \ref{sec3}, we prove a cluster-robust bound for the eigenvalue error which is sharp for the practically most important case $k=1$.  We also establish spectral convergence, which is foundational to all later results.  In Section \ref{sec4} we prove eigenfunction error estimates.  In Section \ref{sec5} we numerically confirm these convergence rates %proved in Section \ref{sec4} 
and investigate the sharpness of the constants in our bounds with respect to spectral properties.   In Section \ref{sec6} we prove superconvergence of eigenvalues and in Section \ref{sec7} provide corresponding numerical results.

\section{Surface Finite Element Method for Eigenclusters} \label{sec2}

\subsection{Weak Formulation and Eigenclusters}
We first define the set 
$$
H^1_\#(\gamma) := \left\{v\in H^1(\gamma): \int_\gamma v ~ d\sigma = 0\right\} \subset H^1(\gamma).
$$
The problem of interest is to find $(u,\lambda)$ satisfying $-\Delta_\gamma u = \lambda u$ with $\int_\gamma u=0$.  The corresponding weak formulation is:  Find $(u,\lambda)\in H^1_\#(\gamma) \times \mathbb{R}^+$ such that
\begin{equation}\label{weig}
\int_{\gamma}\nabla_{\gamma} u \cdot \nabla_{\gamma} v d\sigma
=
\lambda \int_{\gamma}uv ~ d\sigma \qquad \forall v\in H^1_\#(\gamma).
\end{equation}
In order to shorten the notation, we define the bilinear form on $H^1(\gamma)$ and the $L_2$ inner product on $L_2(\gamma)$ respectively as 
\begin{align}\label{d:a}
\tilde a(u,v) &:=\int_{\gamma}\nabla_{\gamma} u \cdot \nabla_{\gamma} v ~ d\sigma,
\\
\label{d:m}
\tilde m(u,v) &:= \int_{\gamma}uv ~ d\sigma.
\end{align}
We equip $H^1(\gamma)$ with the norm $\|.\|_{\tilde a} := \sqrt{\tilde a(.,.)}$.%, which is equivalent to the standard norm on $H^1_\#(\gamma)$. 
We also use the $\tilde m(.,.)$ bilinear form to define the $L_2$ norm on $\gamma$: $\|.\|_{\tilde m} := \sqrt{\tilde m(.,.)}$.  We denote by $\{u_i\}_{i=1}^\infty$ a corresponding orthonormal basis (with respect to $\tilde m(\cdot, \cdot)$) of $H^1_\#(\gamma)$ consisting of eigenfunctions satisfying \eqref{weig}.

We wish to approximate an eigenvalue cluster.  For $n \ge 1$ and $N \ge 0$, we assume  
\begin{equation}\label{a:separation}
\lambda_{n-1} < \lambda_n \qquad \textrm{and} \qquad  \lambda_{n+N} < \lambda_{n+N+1}
\end{equation}
so that the targeted cluster of eigenvalues $\lambda_i$, $i\in J:=\{n, ..., n+N\}$ is separated from the remainder of the spectrum. 

\subsection{Surface approximations} \label{ss:prelim}
\textit{Distance Function.}
We assume that $\gamma$ is a compact, orientable, $C^\infty$,
$D$-dimensional surface without boundary which is embedded in $\mathbb{R}^{D+1}$. Let $d$ be the oriented distance function for $\gamma$ taking negative values in the bounded component of $\mathbb{R}^{D+1}$ delimited by $\gamma$. The outward pointing unit normal of $\gamma$ is then $\bnu := \nabla d$. We denote by $\mathcal{N}\subset \mathbb{R}^{D+1}$ a strip about $\gamma$ of sufficiently small width so that any point $x\in \mathcal{N}$ can be uniquely decomposed as 
\begin{equation}\label{e:proj}
x = \bpsi(x) +d(x)\bnu(x).
\end{equation}
$\bpsi(x)$ is the unique orthogonal projection onto $\gamma$ of $x\in \mathcal N$. We define the projection onto the tangent space of $\gamma$ at $x \in \mathcal N$ as $P(x) := I - \bnu(x)\otimes \bnu(x)$ and the surface gradient satisfies $\nabla_\gamma=P\nabla$.
From now, we assume that the diameter of the strip $\mathcal N$ about $\gamma$ is small enough for the decomposition \eqref{e:proj} to be well defined.

\paragraph{Approximations of $\gamma$}  Multiple options for constructing polynomial approximations of $\gamma$ have appeared. We prove our results under abstract assumptions in order to ensure broad applicability.  Let $\overline{\Gamma}$ be a polyhedron or polytope (depending on $D=\dim(\gamma)$) whose faces are triangles or tetrahedron.  This assumption is made for convenience but is not essential.  The set of all triangular faces of $\overline{\Gamma}$ is denoted $\overline{\mathcal T}$.
  
The higher order approximation $\Gamma$ of $\gamma$ is constructed as follows.  
%Assume we are given a piecewise smooth bi-Lipschitz lift $\bl: \overline{\Gamma} \rightarrow \gamma$.  
Letting $\overline{T} \in \overline{\mathcal{T}}$, we define the degree-$k$ approximation of $\bpsi (\overline{T}) \subset \gamma$ via the Lagrange basis functions $\{\phi_1,...,\phi_{n_k}\}$ with nodal points $\{x^1,...,x^{n_k}\}$ on $\overline{T}$. For $x\in \overline{T}$, we have the discrete projection $\bL: \overline{\Gamma} \to \Gamma$ defined by
\begin{equation}\label{d:interpolated_lift}
\bL(x) := \sum_{j=1}^{n_k}\bL(x^j)\phi_j(x), ~~ \hbox{ where } |\bL(x^j)-\bpsi(x^j)|  \le C h^{k+1}.
\end{equation}
Since we have used the Lagrange basis we have a continuous piecewise polynomial approximation of $\gamma$ which we define as
\begin{equation}\label{e:degreek}
\Gamma := \{\bL(x)\ : \ x\in \overline{\Gamma}\} \qquad \text{and} \qquad \mathcal T := \{ \bL(\overline T) \ : \ \overline{T} \in \overline{\mathcal T} \}.
\end{equation}
The requirement $|\bL(x^j)-\bpsi(x^j)|  \le C h^{k+1}$ ensures good approximation of $\gamma$ by $\Gamma$ while allowing for instances where $\Gamma$ and $\gamma$ do not intersect at interpolation nodes, or even possibly for $\gamma \cap \Gamma = \emptyset$.  This could occur when $\Gamma$ is constructed from imaging data or in free boundary problems.  The assumption \eqref{d:interpolated_lift} also allows for maximum flexibility in constructing $\Gamma$, as we could for instance take $\bL(x^j)=\bl(x^j)$ with $\bl$ a piecewise smooth bi-Lipschitz lift $\bl: \overline{\Gamma} \rightarrow \gamma$ (cf. \cite{MMN11, BCMN:13, BCMMN16}).  

\paragraph{Shape regularity and quasi-uniformity} Associated with a degree-$k$ approximation $\Gamma$ of $\gamma$, we follow \cite{BP:11} and let $\rho:=\rho(\mathcal T)$ be its  shape regularity constant defined as the largest positive real number such that
$$
\rho | \bxi | \leq |D\bF_T(x) \bxi | \leq \rho^{-1} |\bxi|, \qquad \forall \bxi \in \mathbb R^D, \quad \forall T \in \mathcal T \hbox{ and } x \in T, 
$$
where
\begin{equation}
\bF_T := \bL \circ \overline{\bF}_T
\label{effT}
\end{equation}
with $\overline{\bF}_T$ the natural affine mapping from a Kuhn (reference) simplex $\widehat T \subset \mathbb R^D$ to $\overline{T}$.
Further, the quasi-uniform constant $\eta:=\eta(\mathcal T)$ of $\mathcal T$ is the smallest constant such that 
$$
h:=\max_{T\in \mathcal T} \textrm{diam}(T) \leq \eta \min_{T \in \mathcal T} \textrm{diam}(T).
$$

We recall that $\bnu=\nabla d: \mathcal{N} \rightarrow \mathbb{R}^{D+1}$ is the normal vector on $\gamma$ and let $\bN$ be the normal vector on $\Gamma$.  The assumption \eqref{d:interpolated_lift} yields
\begin{align}\label{e:d_estim}
\|d\|_{L_\infty(\Gamma)} & \le Ch^{k+1},
\\ \|\bnu-\bN\|_{L_\infty(\Gamma)} & \le Ch^k,
\\
\label{e:lift_estim}
\|\bL-\bpsi\|_{W^{i,\infty}(\overline{T})} & \le Ch^{k+1-i}, ~~\overline{T} \in \overline{\mathcal{T}},  ~~0 \le i \le k+1,
\end{align}
where $C$ is a constant only depending on $\rho(\mathcal T)$, $\eta(\mathcal T)$ and $\gamma$.

\paragraph{Function Extensions} We assume $\Gamma$ is contained in the strip $\mathcal N$. 
If $\tilde u$ is a function defined on $\gamma$,  we extend it to $\mathcal{N}$ as $u = \tilde u\circ \bpsi$, where $\bpsi$ is defined in \eqref{e:proj}.  Note that $\bpsi |_\Gamma:\Gamma \rightarrow \gamma$ is also a smooth bijection. We can leverage this to relate functions defined on the two surfaces. For a function $u$ defined on $\Gamma$ we define its lift to $\gamma$ as $\tilde u = u \circ\bpsi|_{\Gamma}^{-1}$. 
As a general rule, we use the tilde symbol to denote quantities defined on $\gamma$ but when no confusion is possible, the tilde symbol is dropped.

\paragraph{Bilinear Forms on $\Gamma$}
Given a degree-$k$ approximation $\Gamma$ of $\gamma$, %we define the counterparts of the energy space $H^1_\#(\gamma)$, and bilinear forms $\tilde a(.,.)$ and $\tilde m(.,.)$.
let $H^1_\#(\Gamma):=\{ v \in H^1(\Gamma): \int_{\Gamma} v ~ d\Sigma = 0\} \subset H^1(\Gamma)$
and define the forms on $H^1(\Gamma)$:  
\begin{equation}
A(u,v) :=\int_{\Gamma}\nabla_{\Gamma} u \cdot \nabla_{\Gamma} v ~ d\Sigma,~~~~~M(u,v) :=\int_{\Gamma}uv~ d\Sigma.
\end{equation}
The energy and $L_2$ norms on $\Gamma$ are then $\|.\|_{A} := \sqrt{A(.,.)}$ and $\|.\|_{M} := \sqrt{M(.,.)}$.  

We have already noted that $\bpsi|_\Gamma$ provides a bijection from $\Gamma$ to $\gamma$. Its smoothness (derived from the smoothness of $\gamma$) guarantees that  $H^1(\gamma)$ and $H^1(\Gamma)$ are isomorphic. 
Moreover,  the bilinear form $A(.,.)$ on $H^1(\Gamma)$ can be defined on $H^1(\gamma)$ 
\begin{equation}
\widetilde A(\tilde u ,\tilde v  ) := \int_{\gamma}A_\gamma\nabla_{\gamma} \tilde u  \cdot \nabla_{\gamma} \tilde v ~   d\sigma = \int_{\Gamma}\nabla_{\Gamma} u \cdot \nabla_{\Gamma} v~ d\Sigma = A(u,v)
\label{akdef}
\end{equation}
and similarly for the $L^2$ inner product
\begin{equation}
\widetilde M(\tilde u ,\tilde v  ) :=\int_{\gamma}\tilde u \tilde v  \frac{1}{Q} ~ d\sigma = \int_{\Gamma} uv~ d\Sigma =M(u,v).
\label{mkdef}
\end{equation}
Here $Qd\Sigma = d\sigma$ and $A_\gamma$ depends on the change of variable $\tilde x = \bPsi(x)$.
We refer to \cite{Dziuk, D09} for additional details. 
Again, we use the notations $\| . \|_{\widetilde A}:= \sqrt{\widetilde A(.,.)}$ and $\| . \|_{\widetilde M}:= \sqrt{\widetilde M(.,.)}$.
For the majority of this paper we will work with these lifted forms.

%WWWWWWWWWWWWWWWWWWWWWWWWWWWWWWWWWWWWWWWWWWWWWWWWWWWWWWWWWWWWWWWW
%WWWWWWWWWWWWWWWWWWWWWWWWWWWWWWWWWWWWWWWWWWWWWWWWWWWWWWWWWWWWWWWW
%WWWWWWWWWWWWWWWWWWWWWWWWWWWWWWWWWWWWWWWWWWWWWWWWWWWWWWWWWWWWWWWW
%WWWWWWWWWWWWWWWWWWWWWWWWWWWWWWWWWWWWWWWWWWWWWWWWWWWWWWWWWWWWWWWW
%WWWWWWWWWWWWWWWWWWWWWWWWWWWWWWWWWWWWWWWWWWWWWWWWWWWWWWWWWWWWWWWW
%WWWWWWWWWWWWWWWWWWWWWWWWWWWWWWWWWWWWWWWWWWWWWWWWWWWWWWWWWWWWWWWW
%WWWWWWWWWWWWWWWWWWWWWWWWWWWWWWWWWWWWWWWWWWWWWWWWWWWWWWWWWWWWWWWW

\subsection{Geometric approximation estimates}
The results in this section are essential for estimating effects of approximation of $\gamma$ by $\Gamma$.
Recall that we assume that the diameter of the strip $\mathcal N$ about $\gamma$ is small enough for the decomposition \eqref{e:proj} to be well defined and that $\Gamma \subset \mathcal N$.

The following lemma provides a bound on the geometric quantities $A_\gamma$ and $Q$ appearing in \eqref{akdef} and \eqref{mkdef}; cf. \cite{D09} for proofs.  As we make more precise in Section \ref{sec:sfem}, we write $f \lesssim g$ when $f \le Cg$ with $C$ a nonessential constant.  
\begin{lemma}[Estimates on $Q$ and $A_\gamma$] 	\label{LinfG}
Let $P=I-\bnu \otimes \bnu$ be the projection onto the tangent plane of $\gamma$. Let $A_{\gamma}$ and $Q$ as in \eqref{akdef} and \eqref{mkdef} respectively. Then
	\begin{align}
	\left\|1-1/Q\right\|_{L_\infty(\gamma)} + \|A_\gamma - P\|_{L_\infty(\gamma)}\lesssim h^{k+1}. 
	\end{align}
\end{lemma}

The above geometric estimates along with \eqref{akdef} and \eqref{mkdef} immediately yield estimates for the approximations of $\tilde m(.,.)$ and $\tilde a(.,.)$ by $\widetilde M(.,.)$ and $\widetilde A(.,.)$ respectively.

\begin{corollary}[Geometric estimates]\label{c:geom}
	The following relations hold:
	\begin{align}
	|(\tilde m-\widetilde M)(v,w)|\lesssim h^{k+1}\|v\|_{\tilde m}\|w\|_{\tilde m}, \qquad \forall v,w \in L^2(\gamma)
	\label{m-mk}
\\
	|(\tilde a - \widetilde A)(v,w)|\lesssim h^{k+1}\|v\|_{\tilde a}\|w\|_{\tilde a}, \qquad \forall v,w \in H^1(\gamma).
	\label{a-ak}
	\end{align}
\end{corollary}

The following relations regarding the equivalence of norms are found e.g. in \cite{D09}:
\begin{equation}\label{e:equiv_rough}
\| .\|_{\widetilde A}  \lesssim \| . \|_{\tilde a} \lesssim \| .\|_{\widetilde A} \qquad \textrm{and} \qquad 
\| .\|_{\widetilde M}  \lesssim \| . \|_{\tilde m} \lesssim \| .\|_{\widetilde M}.
\end{equation}
They are valid under the assumption that the diameter of the strip $\mathcal N$ around $\gamma$ is small enough and that $\Gamma \subset \mathcal N$.
We now provide a slight refinement of the above equivalence relations leading to sharper constants. 

\begin{corollary}[Equivalence of norms]
	Assume that the diameter of the strip $\mathcal N$ around $\gamma$ is small enough. There exists constant $C$ only depending on $\gamma$ and on the shape-regularity and quasi-uniformity constants $\rho(\mathcal T)$, $\eta(\mathcal T)$ such that 
	\begin{align}
	\|.\|_{\widetilde A}  &\leq (1+C h^{k+1})\|.\|_{\tilde a}, \qquad  \|.\|_{\tilde a}\leq (1+Ch^{k+1}) \| .\|_{\widetilde A},
	\label{aka}
\\
	\|.\|_{\widetilde M}  &\leq (1+C h^{k+1})\|.\|_{\tilde m}, \qquad \|.\|_{\tilde m} \leq (1+Ch^{k+1})\|.\|_{\widetilde M}.
	\label{mkm}
	\end{align}
\end{corollary}
\begin{proof}
    For brevity, we only provide the proof of \eqref{aka} as the arguments to guarantee \eqref{mkm} are similar and somewhat simpler. 
    Let $v \in H^1(\gamma)$. We have
	\begin{equation}\label{e:Aa}
	\|v\|^2_{\widetilde A} -  \|v\|_{\tilde a}^2 = \widetilde A (v,v) - \tilde a(v,v) = (\widetilde A-\tilde a)(v,v) 
	\end{equation}
	so that in view of the geometric estimate \eqref{aka}, we arrive at
	$$
	\|v\|^2_{\widetilde A} \leq \|v\|_{\tilde a}^2 + |(\widetilde A-\tilde a)(v,v)| \leq (1 + Ch^{k+1})\|v\|^2_{\tilde a}.
	$$
	When $x \ge 0$, the slope of $\sqrt{1+x}$ is greatest at $x=0$ with a value of $\frac{1}{2}$, so $\sqrt{1+x} \leq 1+ \frac{1}{2}x$.  Thus $\sqrt{1+Ch^{k+1}}\leq 1+\frac{1}{2}Ch^{k+1}$, and the first estimate in \eqref{aka} follows by taking a square root.  The remaining estimates are derived similarly.	
	\end{proof}

%WWWWWWWWWWWWWWWWWWWWWWWWWWWWWWWWWWWWWWWWWWWWWWWWWWWWWWWWWWWWWWWW
%WWWWWWWWWWWWWWWWWWWWWWWWWWWWWWWWWWWWWWWWWWWWWWWWWWWWWWWWWWWWWWWW
%WWWWWWWWWWWWWWWWWWWWWWWWWWWWWWWWWWWWWWWWWWWWWWWWWWWWWWWWWWWWWWWW
%WWWWWWWWWWWWWWWWWWWWWWWWWWWWWWWWWWWWWWWWWWWWWWWWWWWWWWWWWWWWWWWW
%WWWWWWWWWWWWWWWWWWWWWWWWWWWWWWWWWWWWWWWWWWWWWWWWWWWWWWWWWWWWWWWW
%WWWWWWWWWWWWWWWWWWWWWWWWWWWWWWWWWWWWWWWWWWWWWWWWWWWWWWWWWWWWWWWW
%WWWWWWWWWWWWWWWWWWWWWWWWWWWWWWWWWWWWWWWWWWWWWWWWWWWWWWWWWWWWWWWW

\subsection{Surface Finite Element Methods}
\label{sec:sfem}

We construct approximate solutions to the eigenvalue problem \eqref{weig} via surface FEM consisting of a finite element methods on degree-$k$ approximate surfaces. See \cite{D09, Dziuk} for more details.

\paragraph{Surface Finite Elements}
Recall that the degree-$k$ approximate surface $\Gamma$ and its associated subdivision $\mathcal T$ are obtained by lifting $\overline{\Gamma}$ and $\overline{\mathcal T}$ via \eqref{e:degreek}.
Similarly, finite element spaces on $\Gamma$ consist of finite element spaces on the (flat) subdivision  $\overline{\mathcal T}$ lifted to $\Gamma$ using the interpolated lift $\bL$ given by \eqref{d:interpolated_lift}.  More precisely, for $r\geq 1$ we set
\begin{equation}\label{e:fem_base}
\mathbb V:=\mathbb{V}(\Gamma,\mathcal T):= \{V \in H^1(\Gamma) \ : \ V = \overline V \circ \bL^{-1}, \ \textrm{with} \ \overline V|_{\overline{T}} \in \mathbb{P}^r(\overline{T}) \quad   \forall \overline{T} \in \overline{\mathcal{T}} \}.
\end{equation}
Here $\mathbb P^r(\overline{T})$ denotes the space of polynomials of degree at most $r$ on $\overline{T}$. 
Its subspace consisting of zero mean value functions is denoted $\mathbb V_\#$:
$$
\mathbb V_\#:= \mathbb V_\#(\Gamma) = \{V \in \mathbb V \ : \ \int_\Gamma V \ d\Sigma = 0\}.
$$

\paragraph{Discrete Formulation}
The proposed finite element formulation of the eigenvalue problem on $\Gamma$ reads: Find $(U,\Lambda)\in \mathbb V_\# \times \mathbb{R}^+$  such that 
\begin{equation}
\label{e:discrete}
A(U,V) = \Lambda ~ M(U,V)\qquad \forall V\in \mathbb V_\#. 
\end{equation}
By the definitions \eqref{akdef}, \eqref{mkdef} of $\widetilde A(.,.)$ and $\widetilde M(.,.)$, relations \eqref{e:discrete} can be rewritten 
$$
\widetilde A(\tilde U,\tilde V) = \Lambda\ \widetilde M( \tilde U,\tilde V) \qquad  \forall V\in \mathbb V_\#.
$$
We denote by $0<  \Lambda_{1} \le ... \le \Lambda_{\dim (\mathbb{V}_\#)}$ and 
$\{U_{1},..., U_{\dim(\mathbb{V}_\#)} \}$  the positive discrete eigenvalues the corresponding $M$-orthonormal discrete eigenfunctions satisfying $M(U_i,1)=0$, $i=1,...,\dim(\mathbb V_\#)$.  
From the definition \eqref{mkdef} of $\widetilde M(.,.)$, $\{ \tilde U_{i}  \}_{i=1}^{\dim(\mathbb V_\#)}$ are pairwise  $\widetilde M-$orthogonal and $\widetilde M(\tilde U_i,1)=0$, for $i=1,...,\dim(\mathbb V_\#)$.

\paragraph{Ritz projection}
We define a Ritz projection for the discrete bilinear form 
$$
\bG: H^1(\gamma)\to \mathbb V_\# 
$$
for any $\tilde v \in H^1(\gamma)$ as the unique finite element function $\bG \tilde v := W \in \mathbb V_\#$ satisfying
\begin{equation}
\widetilde A(\tilde W, \tilde V) = \widetilde A(\tilde v,\tilde  V), \qquad \forall V\in \mathbb V_\#.
\label{GRitz}
\end{equation}

\paragraph{Eigenvalue cluster approximation}

We recall that we target the approximation of an eigencluster indexed by $J$ satisfying the separation assumption \eqref{a:separation}.
We denote the discrete eigencluster and orthonormal basis (with respect to $\widetilde{M}(\cdot, \cdot)$)  by $\{\Lambda_{n}, ..., \Lambda_{n+N}\}\subset \mathbb R^+$ and $\{U_{n},..., U_{n+N}\}\subset \mathbb V_\#$. In addition, we use the notation
$$
\W_\#  := \text{span}\{U_{i}: i \in J\}
$$
to denote the discrete invariant space. We also define the quantity
\begin{equation}\label{d:mu}
\mu(J) := \max_{\ell \in J} \max_{j \notin J} \left | \frac{\lambda_\ell}{\Lambda_{j}-\lambda_\ell} \right |,
\end{equation}
which will play an important role in our eigenfunction estimates. It is finite provided $h$ is sufficiently small, see Remark~\ref{r:mu}.  

\paragraph{Projections onto $\W_\#$}
We denote by $\bP :H^1(\gamma)\to \W_\#$ the $\widetilde M(.,.)$ projection onto $\W_\#$ , i.e., for $\tilde v \in H^1(\gamma)$, $\bP v := W \in \W_\#$ satisfies
$$
\widetilde M(\tilde W,  \tilde V) = \widetilde M(\tilde v,\tilde V),\qquad  \forall V\in \W_\#.
$$

The other projection operator onto $\W_\#$ is defined by
$$
\bZ: H^1(\gamma) \rightarrow \mathbb W_\# \hbox{ s.t. }
\widetilde A(\tilde \bZ \tilde v,\tilde V) = \widetilde A(\tilde v, \tilde V), \qquad \forall V\in\W_\#.
$$
%$$ 
%such that for every $\tilde v \in H^1(\gamma)$
%$$
Notice that $\bZ$ can be thought of as the Galerkin projection onto $\W_\#$, since
\begin{equation}\label{e:bZ}
\bZ\tilde v = \bP ( \widetilde \bG(\tilde v)).
\end{equation}

\paragraph{Alternate surface FEM}  In our analysis of eigenvalue errors we employ a {\it conforming} parametric surface finite element method as an intermediate theoretical tool.  For this, we introduce a finite element space on $\gamma$: 
\begin{equation*}
\widetilde {\mathbb V}:= \{ \tilde V  \ : \ V \in \mathbb V \}.
\end{equation*}
%Note that the space $\mathbb{V}$ is defined on $\Gamma$ and lifted then to $\gamma$, and $\widetilde{ \mathbb{V}}$ thus does not correspond in general to the finite element space defined on $\overline \Gamma$ and lifted to $\gamma$.
The space of vanishing mean value functions (on $\gamma$) is denoted by  $\widetilde{\mathbb V}_\#$:
\begin{equation*}
\widetilde{\mathbb V}_\#:= \{V \in \widetilde{\mathbb V} \ : \ \int_\gamma V ~ d\sigma = 0\}.
\end{equation*}

 For $i=1,...,\dim(\widetilde{\mathbb V}_\#)$, we let $(U^\gamma, \Lambda^{\gamma}_i) \in \widetilde{\mathbb V}_\#\times \mathbb{R}^+$ be finite element eigenpairs computed on the continuous surface $\gamma$, that is, 
\begin{equation}\label{e:eigenproblem_discrete_gamma}
\tilde a(U^\gamma_i, V) = \Lambda_{i}^\gamma \tilde m(U^\gamma_i, V) \qquad \forall V^\gamma \in \widetilde{\mathbb V}_\#.
\end{equation}

\paragraph{Notation and constants} Generally we use small letters ($\gamma$, $u$, $v$,...) to denote quantities lying in infinite dimensional spaces in opposition to capital letters used to denote quantities defined by a finite number of parameters ($\Gamma$, $U$, $V$).  
We also recall that for every function $v:\Gamma \rightarrow \mathbb R$ defines uniquely (via the lift  $\bPsi|_\Gamma$) a function $\tilde v:\gamma \rightarrow \mathbb R$ and conversely. We identify quantities defined on $\gamma$ using a tilde but drop this convention when no confusion is possible, i.e. $v$ could denote a function from $\Gamma$ to $\mathbb R$ as well as its corresponding lift defined from $\gamma$ to $\mathbb R$.

Whenever we write a constant $C$ or $c$, we mean a generic constant that may depend on the regularity properties of $\gamma$ and the Poincar\'e-Friedrichs constant $C_F$ in the standard estimate $\|v\|_{L_2(\gamma)} \le C_F \|v\|_{a}$, $v \in H^1_\#(\gamma)$ and on the shape-regularity $\rho(\mathcal T)$ and quasi-uniformity $\eta(\mathcal T)$ constants, but not otherwise on the spectrum of $-\Delta_\gamma$ and $h$. In addition, by $f \lesssim g$ we mean that $f \le Cg$ for such a nonessential constant $C$.  All other dependencies on spectral properties will be made explicit.  

\section{Clustered Eigenvalue Estimates} \label{sec3}

 Theorem 3.3 of \cite{KO06} gives a cluster-robust bound for cluster eigenvalue approximations in the conforming case.   We utilize this result by employing the conforming surface FEM defined in \eqref{e:eigenproblem_discrete_gamma} as an intermediate discrete problem.  We first use the results of \cite{KO06} to estimate $|\lambda_i - \Lambda^\gamma_{i}|$ in a cluster-robust fashion and then independently bound $|\Lambda^\gamma_{i}-\Lambda_{i}|$. Note that if $\lambda_i$ is a multiple eigenvalue so that $\lambda_{i-\underline{k}}=...=\lambda_i=...=\lambda_{i+\bar k}$, then our bounds also immediately apply to $|\lambda_i-\Lambda_{j}|$,  for $i-\underline{k} \le j \le i+\bar{k}$.  

Because our setting is non-conforming, we introduce two different Rayleigh quotients defined for $v\in \widetilde{\mathbb V}$:
$$
R_{\tilde a}(v)  := \frac{\tilde a(v,v)}{\tilde m(v,v)} \qquad \textrm{and} \qquad R_{\widetilde A}(v)  := \frac{\widetilde A(v,v)}{\widetilde M(v,v)}.
$$
We invoke the min-max approach to characterize the approximate eigenvalues  
\begin{equation}\label{FcEig}
\Lambda^\gamma_{j}  = \min_{\substack{\mathbb S\subset \widetilde{\mathbb V} \\ \dim(\mathbb S)=j+1}} \max_{V\in \mathbb S} R_{\tilde a}(V) \qquad \text{and} \qquad  \Lambda_{j} = \min_{\substack{\mathbb S\subset \widetilde{\mathbb V}\\ \dim(\mathbb S)=j+1}} \max_{V\in \mathbb S} R_{\widetilde A}(V).
\end{equation}
Notice that we do not restrict the Rayleigh quotients to functions with vanishing mean values.  Thus we consider subspaces of dimensions $\dim(S)=j+1$ rather than the usual $\dim(S)=j$. The extra dimension is the space of constant functions.   

The bound for $|\Lambda^\gamma_{j} - \Lambda_{j}|$ given in the following lemma shows that this difference is  only related to the geometric error  scaled by  the corresponding exact eigenvalue $\Lambda^\gamma_j$.
\begin{lemma}
\label{eig_compare}
For $i=1,...,\dim(\mathbb V)-1$, let $\Lambda^\gamma_{i}$ and $\Lambda_{i}$ be the discrete eigenvalues associated with the finite element method on $\gamma$ and $\Gamma$ respectively.
 Then, we have
\begin{equation}\label{eigcompare}
|\Lambda^\gamma_{i} - \Lambda_{i}| \lesssim \Lambda^\gamma_{i}h^{k+1}.
\end{equation}
\end{lemma}	
\begin{proof}
We use the characterization \eqref{FcEig} and compare $R_a(.)$ and $R_{\ak}(.)$.
Using the finer norm equivalence properties \eqref{aka} and \eqref{mkm}, we have for $V\in \widetilde{\mathbb V}$
$$
R_{\widetilde A}(V) \leq \frac{(1+Ch^{k+1})^2 \tilde a(V,V)}{\tilde m(V,V)/(1+Ch^{k+1})^2} = (1+Ch^{k+1})^4R_{\tilde a}(V).
$$
Thus
\begin{align}
\begin{aligned}\label{e:LL}
\Lambda_{i} & \leq  \min_{\substack{\mathbb S\subset \mathbb V \\ \dim(\mathbb S)=i+1}} \max_{V\in \mathbb S} (1+Ch^{k+1})^4R_{\tilde a}(V)
=(1+Ch^{k+1})^4\Lambda^\gamma_{i},
\\ \Lambda_{i} -\Lambda^\gamma_{i} &  \lesssim \Lambda^\gamma_{i}h^{k+1}.
\end{aligned}
\end{align}
A similar argument gives $\Lambda_{i}^\gamma -\Lambda_{i}\lesssim \Lambda_i h^{k+1} \lesssim \Lambda_i^\gamma h^{k+1}$, where we used \eqref{e:LL} in the last step.
This implies \eqref{eigcompare}, as claimed.  
\end{proof}

We now translate Theorem 3.3 of \cite{KO06} into our notation in order to bound $|\lambda_i-\Lambda^\gamma_{j}|$ in a cluster-robust manner.  First, let ${\bG}^\gamma$ be the Ritz projection calculated with respect to $\tilde a(\cdot, \cdot)$.  That is, for $v\in H^1(\gamma)$,  ${\bG}^\gamma v \in \widetilde{\mathbb{V}}_\#$ satisfies 
$$
\tilde a(\bG^\gamma v, V) = \tilde a(v, V), \qquad \forall V \in \widetilde{\mathbb{V}}_\#.
$$  
Next, let $T: H^1_\#(\gamma) \rightarrow H^1_\#(\gamma)$ be the solution operator associated with the source problem (restricted to $H^1_\#(\gamma)$)
$$
\tilde a(Tf, v) =\tilde m(f,v), \qquad \forall v \in H^1_\#(\gamma).
$$
Finally, let $\bZ^\gamma_n$ be the $\tilde a$-orthogonal projection onto the space spanned by\\
 $\{ U^\gamma_i\}_{i=1,.., n-1}$, that is, onto the first $n-1$ discrete eigenfunctions calculated with respect to $\tilde a$ and $\tilde m$, see \eqref{e:eigenproblem_discrete_gamma}.  Theorem 3.3 of \cite{KO06} provides the following estimates.  

\begin{lemma}[Theorem 3.3 of \cite{KO06}]
Let $j \in J$, and assume that 
\begin{align}
\min_{i=1,..., n-1} |\Lambda_{i}^\gamma-\lambda_j| \neq 0.
\end{align}
Then,
\begin{equation*}
0 \le \frac{\Lambda^\gamma_{j} - \lambda_j}{\lambda_j}  \le \left ( 1 + \max_{i=1,.., n-1} \frac{(\Lambda^\gamma_{i})^2 \lambda_j^2}{|\Lambda^\gamma_{i}-\lambda_j|^2} \sup_{\substack{v \in H^1_\# (\gamma)\\ \| v \|_{\tilde a}=1}} \|(I-{\bG}^\gamma) T \bZ^\gamma_n v \|_{\tilde a}^2 \right ) 
\end{equation*}
$$
\times\sup_{\substack{w \in \textrm{span}(u_k \ : \ k \in J)\\ \|w\|_{\tilde a}=1} } \|(I-\bG^\gamma)w\|_{\tilde a}^2.
$$
\end{lemma}  

We now provide some interpretation of this result.  
Because $\bG^\gamma$ is the Ritz projection defined with respect to $\tilde a(\cdot, \cdot)$, we have
\begin{align}
\|(I-\bG^\gamma) v\|_{\tilde a} = \inf_{V \in \widetilde{\mathbb{V}}_\#} \|v-V\|_{\tilde a}. 
\end{align}
That is, the term $\sup_{w \in \textrm{span}(u_k \ : \ k \in J), \|w\|_{\tilde a}=1 } \|(I-\bG^\gamma)w\|_{\tilde a}^2$ measures approximability in the energy norm of the eigenfunctions in the targeted cluster $ \textrm{span}(u_k \ : \ k \in J)$ by the finite element space.  

Next, we unravel the term $\|(I-{\bG}^\gamma) T \bZ^\gamma_n v \|_{\tilde a}$.  For $v \in H^1_\#(\gamma)$, we have $\bZ^\gamma_n v \in \widetilde{\mathbb V}_\# \subset H^1_\#(\gamma)$.   Because $\gamma$ is assumed to be smooth, a standard shift theorem guarantees that for $f:=\bZ^\gamma_n v \in H^1_\#(\gamma)$, $Tf \in H^3(\gamma) \cap H^1_\#(\gamma)$ and $\|Tf\|_{H^{3}(\gamma)} \lesssim \|f\|_{H^1(\gamma)}$.  
Thus, $T \bZ^\gamma_n v \in H^3(\gamma)$, and $\|T {\bZ}^\gamma_n v\|_{H^3(\gamma)} \lesssim \| v \|_{H^1(\gamma)}$. 
Therefore, $\|(I-\bG^\gamma) T \bZ^\gamma_n v \|_{\tilde a}$ measures the Ritz projection error of $v \in H^3(\gamma)$ in the energy norm, and so (cf. \cite{D09})
\begin{align}
\sup_{v \in H^1_\# (\gamma), \ \| v \|_{\tilde a}=1} \|(I-{\bG}^\gamma) T \bZ^\gamma_n v \|_{\tilde a} \lesssim h^{\min\{2,r\}}.
\end{align}
Combining the previous two lemmas with these observations yields the following.

\begin{theorem}[Cluster robust estimates]
Let $j \in J$, and assume in addition that $\min_{i=1,..., n-1} |\Lambda^\gamma_i-\lambda_j| \neq 0$.  Then
\begin{align}
\label{eig_est}
\begin{aligned}
|\lambda_j - \Lambda_{j}|    &\lesssim \Lambda^\gamma_j \left (1 + C h^{\min\{2r, 4\}} \max_{i=1,.., n-1} \frac{(\Lambda^\gamma_i)^2 \lambda_j^2}{|\Lambda^\gamma_{i}-\lambda_j|^2} \right) \\
&\times\sup_{\substack{w  \in \textrm{span}(u_k \ : \ k \in J)\\ \|w\|_{\tilde a}=1}} \inf_{V \in \widetilde{\mathbb V}_\#}\|w-V\|_{\tilde a}^2 + C h^{k+1} \Lambda^\gamma_j.
\end{aligned}
\end{align}
\end{theorem}
\begin{remark}[Asymptotic nature of eigenvalue estimates]\label{r:mu}
{\it The constant\\ 
	$\max_{i=1,..., n-1} \frac{\Lambda^\gamma_i \lambda_j}{|\Lambda^\gamma_i-\lambda_j|}$ is not entirely a priori and could be undefined if by coincidence $\Lambda^\gamma_{i}-\lambda_j=0$ for some $i<n$.  Because this constant arises from a conforming finite element method, however, its properties are well understood; cf. \cite[Section 3.2]{KO06} for a detailed discussion.  In short, convergence of the eigenvalues $\Lambda^\gamma_{i} \rightarrow \lambda_i$ is guaranteed as $h \rightarrow 0$, so $\max_{i=1,..., n-1} \frac{\Lambda^\gamma_{i} \lambda_j}{|\Lambda^\gamma_{i}-\lambda_j|} \rightarrow \frac{\lambda_{n-1} \lambda_j}{|\lambda_{n-1}-\lambda_j|}$.  Because $j \ge n$ and we have assumed separation property \eqref{a:separation}, namely  $\lambda_n> \lambda_{n-1}$, this quantity is well-defined.  

In the following section we prove eigenfunction error estimates under the assumption that the quantity $\mu(J)= \max_{\ell \in J} \max_{j \notin J} \left | \frac{\lambda_\ell}{\Lambda_{h,j}-\lambda_\ell} \right |$ defined in \eqref{d:mu} above is finite.  The observation in the preceding paragraph and \eqref{eig_est} guarantee the existence of $h_0$ such that $\mu(J)<\infty$ for all $h \le h_0$.  Thus there exists $h_0$ such that for all $h \le h_0$ the discrete eigenvalue cluster respects the separation of the continuous cluster from the remainder of the spectrum in the sense that $\Lambda_{n}>\lambda_{n-1}$ and $\Lambda_{n+N} < \lambda_{n+N+1}$.  
}
\end{remark}
\begin{remark}[Constant in \eqref{eig_est}]
{\it The spectrally dependent constants in \eqref{eig_est} are expressed with respect to the intermediate discrete eigenvalues $\Lambda^\gamma_{j}$ instead of with respect to the computed discrete eigenvalues $\Lambda_{j}$.  It is not difficult to essentially replace $\Lambda^\gamma_{j}$ by $\Lambda_{j}$ at least for $h$ sufficiently small by noting that Lemma \ref{eig_compare} may be rewritten as $|\Lambda_{j}-\Lambda^\gamma_{j}| \lesssim \Lambda_{j} h^{k+1}$.  We do not pursue this change here.}
\end{remark}

%WWWWWWWWWWWWWWWWWWWWWWWWWWWWWWWWWWWWWWWWWWWWWWWWWWWWWWWWWWWWWWWW
%WWWWWWWWWWWWWWWWWWWWWWWWWWWWWWWWWWWWWWWWWWWWWWWWWWWWWWWWWWWWWWWW
%WWWWWWWWWWWWWWWWWWWWWWWWWWWWWWWWWWWWWWWWWWWWWWWWWWWWWWWWWWWWWWWW
%WWWWWWWWWWWWWWWWWWWWWWWWWWWWWWWWWWWWWWWWWWWWWWWWWWWWWWWWWWWWWWWW
%WWWWWWWWWWWWWWWWWWWWWWWWWWWWWWWWWWWWWWWWWWWWWWWWWWWWWWWWWWWWWWWW
%WWWWWWWWWWWWWWWWWWWWWWWWWWWWWWWWWWWWWWWWWWWWWWWWWWWWWWWWWWWWWWWW
%WWWWWWWWWWWWWWWWWWWWWWWWWWWWWWWWWWWWWWWWWWWWWWWWWWWWWWWWWWWWWWWW

\section{Eigenfunction Estimates} \label{sec4}
\subsection{$L_2$ Estimate}

We start by bounding the difference between the Galerkin projection $\bG$ of an exact eigenfunction and its projection to the discrete invariant space. 
It is instrumental for deriving  $L^2$ and energy bounds (Theorems~\ref{t:L2} and \ref{t:energy}).

\begin{lemma}\label{l:L2}
Let $\{\lambda_j\}_{j\in J}$ be an exact eigenvalue cluster satisfying the separation assumption \eqref{a:separation}.  Let $\{\Lambda_{j}\}_{j=1}^{\dim (\mathbb{V}_\#)}$ be the set of approximate FEM eigenvalues satisfying $\mu(J)<\infty$, where $\mu(J)$ is defined in \eqref{d:mu}.
Fix $i \in J$ and  let $u_i \in H^1_\#(\gamma)$ be any eigenfunction associated with $\lambda_i$.
Then for any $\alpha \in \mathbb{R}$, there holds
\begin{equation}\label{L2L}
\|\bG u_i - \bZ  u_i\|_{\widetilde M} \lesssim \left(1 +\mu(J)\right)(\| u_i- \bG u_i-\alpha\|_{\widetilde M} + h^{k+1}\|u_i\|_{\widetilde M}).
\end{equation}
\end{lemma}
\begin{proof} Our proof essentially involves accounting for geometric variational crimes in an argument given for the conforming case in \cite{CaGe11} (cf. \cite{Gallistl}).  

\step{1} Recall that $\{U_j\}_{j=1}^{\dim(\mathbb V_\#)} \in \mathbb V_\#$ denotes the collection of discrete $\widetilde M$-orthonormal eigenfunctions associated with $\{ \Lambda_j\}_{j=1}^{\dim(\mathbb V_\#)}$.
For $l \in \{1,...,\dim(\mathbb V_\#)\} \setminus J$, $U_{l} \in \text{Ran} (I -\bP) \subset \mathbb V_\#$ is $\widetilde M$-orthogonal to the approximate invariant space $\mathbb W_\# =\textrm{span}( U_j \ : \ j \in J \}$. 
According to relation \eqref{e:bZ}, we then have  $\widetilde M(\bZ  u_i,U_{l})  = \widetilde M(\bP \bG u_i,U_{l}) = 0,$
which implies
\begin{equation}\label{e:orth}
\widetilde M(\bG u_i - \bZ u_i, U_{l})  = \widetilde M(\bG u_i, U_{l}).
\end{equation}
In addition, $W := \bG u_i - \bZ u_i = (I-\bP)\bG u_i$ can be written as 
$
W = \sum_{\substack{l=1\\ l \not \in J}}^{\dim(\mathbb V_\#)} \beta_l U_{l}
$
for some $\beta_l \in \mathbb R$, so that, together with \eqref{e:orth}, we have
\begin{equation}\label{e:start}
\|W\|_{\widetilde M}^2 = \widetilde M(W,W) = \widetilde M\left(\bG u_i, \sum_{\substack{l=1\\l\not \in J}}^{\dim(\mathbb V_\#)} \beta_l U_{l}\right).
\end{equation}

\step{2} We now proceed by deriving estimates for $\widetilde M\left(\bG u_i, U_{l}\right)$, $l \not \in J$.
Since $U_{l}$ is an eigenfunction of the approximate eigenvalue problem associated with $\Lambda_l$, we have
$$
\Lambda_{l}\widetilde M(V, U_{l}) = \Lambda_l \widetilde M(U_l,V) = \widetilde A(U_l,V) = \widetilde A(V,U_l) ,\qquad \forall V\in\mathbb V_\#. 
$$
Choosing $V = \bG u_i$  gives
$$
\Lambda_{l}\widetilde M(\bG u_i, U_{l}) = \widetilde A(\bG u_i, U_{l}) 
= \widetilde A (u_i,U_{l}) = \tilde a(u_i,U_{l}) + (\widetilde A-\tilde a)(u_i,U_{l}).
$$
We now use the fact that $u_i$ is an eigenfunction of the exact problem to get
\begin{align*}
\Lambda_{l}\widetilde M(\bG u_i, U_{l})  &= \lambda_i \tilde m(u_i,U_{l}) + (\widetilde A-\tilde a)(u_i,U_{l}) \\
& = \lambda_i\widetilde M(u_i,U_{l}) + \lambda_i (\tilde m-\widetilde M)(u_i,U_{l}) + (\widetilde A-\tilde a)(u_i,U_{l}).
\end{align*}

Subtracting $\lambda_i\widetilde M(\bG u_i,U_{l})$ from both sides yields
$$
(\Lambda_{l}-\lambda_i)\widetilde M(\bG u_i, U_{l}) = \lambda_i\widetilde M(u_i-\bG u_i,U_{l}) + \lambda_i (\tilde m-\widetilde M)(u_i,U_{l}) + (\widetilde A-\tilde a)(u_i,U_{l}),
$$
or
\begin{align*}
\widetilde M(\bG u_i, U_{l}) &= \frac{1}{\Lambda_{l}-\lambda_i}\left [ \lambda_i \widetilde M(u_i-\bG u_i,U_{l}) + \lambda_i (\tilde m-\widetilde M)(u_i,U_{l}) + (\widetilde A-\tilde a)(u_i,U_{l}) \right].
\end{align*}

\step{3} Returning to \eqref{e:start}, we obtain
\begin{align*}
\|W\|_{\widetilde M}^2 = \widetilde M(W,W) &=  \widetilde M\left(u_i-\bG u_i -\alpha,\sum_{\substack{l=1\\l\not \in J}}^{\dim(\mathbb V_\#)} \frac{\lambda_i}{\Lambda_{l}-\lambda_i}\beta_l U_{l}\right) \\
&+  \left [(\tilde m-\widetilde M) + \frac{1}{\lambda_i} (\widetilde A-\tilde a) \right ]\left(u_i,\sum_{\substack{l=1\\ l \not \in J}}^{\dim(\mathbb V_\#)} \frac{\lambda_i}{\Lambda_{l}-\lambda_i}\beta_l U_{l}\right) 
\end{align*}
where we used $\widetilde M(U_l,1)=0$ to incorporate $\alpha \in \mathbb R$ into the estimate.
To continue further, we use the orthogonality property of the discrete eigenfunctions to obtain
$$
\left \| \sum_{\substack{l=1\\ l \not \in J}}^{\dim(\mathbb V_\#)} \frac{\lambda_i}{\Lambda_{l}-\lambda_i}\beta_l U_{l} \right \|_{\widetilde M}^2 = \sum_{\substack{l=1\\ l \not \in J}}^{\dim(\mathbb V_\#)} \left(\frac{\lambda_i}{\Lambda_{l}-\lambda_i}\right)^2\beta_l^2 \| U_l \|_{\widetilde M}^2\leq  \mu(J) \| W \|_{\widetilde M}^2
$$
and similarly $
\left\| \sum_{\substack{l=1\\ l \not \in J}}^{\dim(\mathbb V_\#)} \frac{\lambda_i}{\Lambda_{l}-\lambda_i}\beta_l U_{l} \right\|_{\widetilde A}^2 \leq   \mu(J) \| W \|_{\widetilde A}^2$ since $\tilde A(U_l,U_k) = \Lambda_l \widetilde M(U_l,U_k)$.  
Thus the geometric error estimates (Corollary~\ref{c:geom}) and a Young inequality imply
\begin{align}
\begin{aligned}
\|W\|_{\widetilde M}^2 &\leq  \mu(J)\|u_i-\bG u_i-\alpha\|_{\widetilde M}\|W\|_{\widetilde M}+  C h^{k+1}\mu (J)\|u_i\|_{\widetilde M}\|W\|_{\widetilde M}\\
&+ C h^{2k+2}\frac{\mu(J)^2}{\lambda_i}\|u_i\|_{\widetilde A}^2 
+ \frac{1}{4\lambda_i}\|W\|^2_{\widetilde A}.
\end{aligned}
\label{vmk2}
\end{align}

\step{4} To bound $\|W\|_{\tilde A}$, we recall that $\bP\circ \bG$ and $\bG$ are the $\widetilde A(\cdot, \cdot)$ projections onto $\mathbb{W}_\#$ and $\mathbb{V}_\#$, respectively, and that ${\bP}$ is the $L_2$ projection onto $\mathbb{W}_\#$.  Thus 
\begin{align*}
\|W\|_{\widetilde A}^2 &= \widetilde A(W,W) = \widetilde A((I-\bP)\bG u_i,(I-\bP)\bG u_i) = \widetilde A(\bG u_i, (I-\bP)\bG u_i) \\
&= \widetilde A(u_i,(I-\bP)\bG u_i)= \widetilde A(u_i,W).
\end{align*}

To isolate the geometric error, we rewrite for any $\alpha \in \mathbb R$ the right hand side of the above equation as
\begin{align*}
&\tilde a(u_i,W) + (\widetilde A-\tilde a)(u_i,W) = \lambda_i \tilde m(u_i,W) + (\widetilde A-\tilde a)(u_i,W) \\
& = \lambda_i (\tilde m-\widetilde M)(u_i,W) + \lambda_i \widetilde M(u_i-\bG u_i,W) + \lambda_i \widetilde M(\bG u_i-\bZ u_i,W) + (\widetilde A-\tilde a)(u_i,W)\\
& = \lambda_i (\tilde m-\widetilde M)(u_i,W) + \lambda_i \widetilde M(u_i-\bG u_i-\alpha,W) + \lambda_i \widetilde M(W,W) + (\widetilde A-\tilde a)(u_i,W),
\end{align*}
upon invoking the orthogonality relations \eqref{e:orth} and $\widetilde M(W,1)=0$. 
We take advantage again of the geometric error estimates (Corollary~\ref{c:geom}) to arrive at
\begin{align}
\begin{aligned}
\|W\|_{\widetilde A}^2 &\leq \lambda_i C h^{k+1}\|u_i\|_{\widetilde M}\|W\|_{\widetilde M} + \lambda_i \|u_i-\bG u_i-\alpha \|_{\widetilde M} \|W\|_{\widetilde M} +\lambda_i\|W\|_{\widetilde M}^2\\
&+ C h^{k+1}\|u_i\|_{\widetilde A}\|W\|_{\widetilde A}.
\end{aligned}
\label{vak0}
\end{align}

Now, noting that  $\|u_i\|_{\widetilde A}\lesssim \|u_i\|_{\tilde a}= \sqrt{\lambda_i}\|u_i\|_{\tilde m}$ by \eqref{e:equiv_rough}   and using Young's inequality to absorb the last term by the left hand side gives
\begin{align}
\begin{aligned}
\|W\|_{\widetilde A}^2 &\leq  C h^{k+1}\lambda_i\|u_i\|_{\widetilde M}\|W\|_{\widetilde M} + 2 \lambda_i\|u_i-\bG u_i-\alpha\|_{\widetilde M} \|W\|_{\widetilde M} +2\lambda_i \|W\|_{\widetilde M}^2\\
& + C\lambda_i h^{2k+2} \|u_i\|_{\widetilde M}^2.
\end{aligned}
\label{vak}
\end{align}

\step{5} Using \eqref{vak} in \eqref{vmk2} gives
\begin{align*}
\|W\|_{\widetilde M}^2 & \leq \left(\frac{1}{2} +\mu(J)\right)\|u_i-\bG u_i-\alpha\|_{\widetilde M}\|W\|_{\widetilde M}
+  Ch^{k+1}\left(1+\mu(J)\right)\|u_i\|_{\widetilde M}\|W\|_{\widetilde M}\\
& 
+ Ch^{2k+2}\left(1+\mu(J)^2\right)\|u_i\|_{\widetilde M}^2 + \frac{1}{2}\|W\|_{\widetilde M}^2.
\end{align*}
We apply Young's inequality again to arrive at
\begin{align*}
\|W\|_{\widetilde M}^2 &\lesssim (1 +\mu(J))^2 \left [ \|u_i-\bG u_i-\alpha\|_{\widetilde M}^2
+  h^{2k+2} \|u_i\|_{\widetilde M}^2+ h^{2k+2} \|u_i\|_{\mk}^2 \right ],
\end{align*}
which yields the desired result upon taking a square root.
\end{proof}

\begin{theorem}[$L^2$ error estimate]\label{t:L2}
Let $\{\lambda_j\}_{j\in J}$ be an exact eigenvalue cluster satisfying the separation assumption \eqref{a:separation}. Let $\{\Lambda_{j}\}_{j=1}^{\dim (\mathbb{V}_\#)}$ be the set of approximate FEM eigenvalues satisfying $\mu(J)<\infty$.
We fix $i \in J$ and denote by $u_i \in H^1_\#(\gamma)$ any eigenfunction associated with $\lambda_i$.
Then for any $\alpha \in \mathbb{R}$, the following bound holds:
\begin{align}
\label{l2_bound}
\begin{aligned}
\|u_i - \bP u_i-\alpha\|_{\widetilde M} & \leq \|u_i - \bZ u_i-\alpha \|_{\widetilde M} 
\\ &\lesssim (1 +\mu(J))\|u_i-\bG u_i-\alpha\|_{\widetilde M}+ \left(1 +\mu(J)\right)\|u_i\|_{\widetilde M}h^{k+1}.
\end{aligned}
\end{align}
\end{theorem}

\begin{proof}
Because $\bP \alpha =\bZ \alpha=0$ and $\bP$ is the $\widetilde M$-projection onto $\mathbb{W}_\#$, we have 
$$
\begin{aligned}
\|(u_i-\alpha)-\bP u_i\|_{\widetilde M}& =\|u_i-\alpha-\bP(u_i-\alpha)\|_{\widetilde M} \le \|(u_i-\alpha)-\bZ (u_i-\alpha)\|_{\widetilde M}
\\ & =\|u_i-\bZ u_i -\alpha\|_{\widetilde M}  \leq \|u_i - \bG u_i-\alpha\|_{\widetilde M} + \|\bG u_i - \bZ u_i\|_{\widetilde M}.
\end{aligned}
$$
The second leg is bounded using Lemma~\ref{l:L2}.
\end{proof}

%WWWWWWWWWWWWWWWWWWWWWWWWWWWWWWWWWWWWWWWWWWWWWWWWWWWWWWWWWWWWWWWW
%WWWWWWWWWWWWWWWWWWWWWWWWWWWWWWWWWWWWWWWWWWWWWWWWWWWWWWWWWWWWWWWW
%WWWWWWWWWWWWWWWWWWWWWWWWWWWWWWWWWWWWWWWWWWWWWWWWWWWWWWWWWWWWWWWW
%WWWWWWWWWWWWWWWWWWWWWWWWWWWWWWWWWWWWWWWWWWWWWWWWWWWWWWWWWWWWWWWW
%WWWWWWWWWWWWWWWWWWWWWWWWWWWWWWWWWWWWWWWWWWWWWWWWWWWWWWWWWWWWWWWW
%WWWWWWWWWWWWWWWWWWWWWWWWWWWWWWWWWWWWWWWWWWWWWWWWWWWWWWWWWWWWWWWW
%WWWWWWWWWWWWWWWWWWWWWWWWWWWWWWWWWWWWWWWWWWWWWWWWWWWWWWWWWWWWWWWW

\subsection{Energy Estimate}
We now focus on estimates for  $\| u_i- \bZ u_i\|_{\widetilde A}$.
\begin{theorem}[Energy estimate]\label{t:energy}
Let $\{\lambda_j\}_{j\in J}$ be an exact eigenvalue cluster satisfying the separation assumption \eqref{a:separation}. Let $\{\Lambda_{j}\}_{j=1}^{\dim (\mathbb{V}_\#)}$ be a set of approximate FEM eigenvalues satisfying $\mu(J)<\infty$.
We fix $i \in J$ and denote by $u_i \in H^1_\#(\gamma)$ any eigenfunction associated with $\lambda_i$.
Then for any $\alpha \in \mathbb{R}$, the following bound holds:
\begin{align}
\label{energy_bound}
\begin{aligned}
\| u_i- \bZ u_i\|_{\widetilde A}&\leq \|u_i - \bG u_i\|_{\widetilde A} + C \sqrt{\lambda_i} (1+\mu(J)) \|u_i- \bG u_i-\alpha\|_{\widetilde M}\\
&+ C \sqrt{\lambda_i} (1+\mu(J)) h^{k+1} \|u_i\|_{\widetilde M}. 
\end{aligned}
\end{align}

\end{theorem}	

\begin{proof}
Let $W := \bG u_i - \bZ u_i$.  We restart from the estimate  \eqref{vak} for $\| W \|_{\widetilde A}$, apply Young's inequality, and take advantage of the $L^2$ error bound \eqref{L2L} to deduce
$$
\begin{aligned}
\|W\|_{\widetilde A}^2  & \lesssim   \lambda_i ( h^{2k+2}\|u_i\|_{\widetilde M}^2 + \|u_i-{\bG} u_i-\alpha\|_{\widetilde M}^2 + \|W\|_{\widetilde M}^2)
\\ & \lesssim  \lambda_i (1+\mu(J))^2 (h^{2k+2}\|u_i\|_{\widetilde M}^2 + \|u_i-{\bG} u_i-\alpha\|_{\widetilde M}^2).  
\end{aligned}
$$
The desired result follows from 
$\|u_i-\bZ u_i\|_{\widetilde A} \leq \|u_i-\bG u_i\|_{\widetilde A} +\|W\|_{\widetilde A}.$
\end{proof}

We end by commenting on \eqref{energy_bound}.  Because ${\bf G}$ is the Galerkin projection onto $\mathbb{V}_\#$ with respect to $\widetilde A(\cdot, \cdot)$, we have for the first term in \eqref{energy_bound} that
\begin{align}
\|u_i-{\bf G}u_i\|_{\widetilde A}  \le \inf_{V \in \mathbb{V}_\#} \|u_i-V\|_{\widetilde A} %= \inf_{V \in \mathbb{V}_\#; \ \alpha \in \mathbb R} \|u_i-V-\alpha \|_{\widetilde A}
= \inf_{V \in \mathbb{V}} \|u_i-V \|_{\widetilde A}.
\end{align}
Here we used that $\widetilde A(\tilde v,1)=0$, $v \in H^1(\gamma)$.
The last term above may be bounded in a standard way (cf. \cite{CD15} for definition of a suitable interpolation operator of Scott-Zhang type in any space dimension).  Similar comments apply to \eqref{l2_bound}.

Bounding $\|u_i-{\bf G}\|_{\widetilde M}$ is more complicated.  Because $\Gamma$ is not smooth, it is not possible to directly carry out a duality argument to obtain $L_2$ error estimates for ${\bf G}$ with no geometric error term.  Abstract arguments of \cite{D09} however give error bounds for $u_i-{\bG} u_i$ satisfying $\tilde a(u_i-{\bG}u_i, V) =F(V) ~ \forall V \in \mathbb V_\#$.  Letting  $F(V)= (\tilde a-\widetilde A)(u_i-{\bG} u_i, V)$, the fact that $\widetilde A(\tilde v,1)=0$ for any $v \in H^1(\gamma)$ yields $$
\tilde a(u_i-{\bG}u_i, V) =F(V)\qquad \forall V \in \mathbb V.
$$
Choosing $\alpha= \frac{1}{|\gamma|} \int_\gamma {\bG} (u-u_i)$,  \cite[Theorem 3.1]{D09} along with \eqref{a-ak}   then yield
$$
\|u_i-{\bG} u_i- \alpha \|_{\tilde m}  \lesssim h \min_{V \in \mathbb{V}} \|u_i-V\|_{\tilde a} + h^{k+1} \|u_i -{\bG} u_i\|_{\tilde a}
  \lesssim  h \min_{V \in \mathbb{V}} \|u_i-V\|_{\widetilde A}.
$$
Thus the $L_2$ term above may also be bounded in a standard way.

\subsection{Relationship between projection errors}  Many classical papers on finite element eigenvalue approximations contain energy error bounds for the projection error $\|v-\bP v\|_{\tilde a}$  \cite{BO89, BO91}.  We briefly investigate the relationship between this error notion and our notion $\|v-\bZ v\|_{\tilde a}$.  Because ${\bZ}$ is a Galerkin projection, we have $\|v-{\bZ} v\|_{\widetilde A} \le \|v-\bP v\|_{\widetilde A}$.  In Proposition \ref{prop:equiv} we show that the reverse inequality holds up to higher-order terms.  These two error notions are thus asymptotically equivalent.  
\begin{lemma}\label{Plesssimon}
Let $\{\lambda_j\}_{j\in J}$ be an exact eigenvalue cluster indexed by $J$  satisfying the separation assumption \eqref{a:separation}. Let $\{\Lambda_{j}\}_{j=1}^{\dim (\mathbb{V}_\#)}$ be set of approximate FEM eigenvalues satisfying $\mu(J)<\infty$.
We assume that for an absolute constant $B$, there holds $\max\{\Lambda_{n+N}\} \le B.$
Then for $v \in H^1(\gamma)$, we have
$$
\|\bP v\|_{\widetilde A} \le \sqrt{B} \| v\|_{\widetilde M}.
$$
\end{lemma}
\begin{proof}
Since $\bP v \in \mathbb W_\#$, there exists $\beta_j$, $j\in J$, such that
$\bP v=  \sum_{j \in J}\beta_j U_{j}.$
Thus
\begin{align*}
\|\bP v\|_{\widetilde A}^2  & = \widetilde A(\bP v, \bP v) = \sum_{j \in J} \beta_j \widetilde A(U_{j},\bP v) = \sum_{j\in J}\beta_j \Lambda_{j}\widetilde M(U_{j},\bP v)
\\ 
& =\sum_{j \in J}\beta_j \Lambda_{j}\widetilde M(U_{j},\sum_{j\in J}\beta_j U_{j}) = \sum_{j\in J}\beta_j^2 \Lambda_{j} \widetilde M(U_j,U_j) \leq B \|\bP v\|_{\widetilde M}^2 \le B\|v\|_{\widetilde M}^2,
\end{align*}
where we used that the discrete eigenfunctions $\{ U_j \}$ are $\widetilde M$-orthogonal.
\end{proof}

\begin{proposition}
\label{prop:equiv}
Let $\{\lambda_j\}_{j\in J}$ be an exact eigenvalue cluster indexed by $J$  satisfying the separation assumption \eqref{a:separation}. Let $\{\Lambda_{j}\}_{j=1}^{\dim (\mathbb{V}_\#)}$ be set of approximate FEM eigenvalues satisfying $\mu(J)<\infty$.
Furthermore, assume that for some absolute constant $B$, $\Lambda_{N+n} \leq B.$
Let $u_i$ be an eigenfunction with eigenvalues $\lambda_i$, for some $i \in J$. 
Then the following bound holds for any $\alpha \in \mathbb{R}$:
$$
\|u_i-\bP u_i\|_{\widetilde A} \leq \|u_i-\bZ u_i\|_{\widetilde A} + \sqrt{B}\|u_i-  \bG u_i-\alpha\|_{\widetilde M}.
$$
\end{proposition}

\begin{proof}
By the triangle inequality we have:
$$
\|u_i-\bP u_i\|_{\widetilde A} \leq \|u_i-\bZ u_i\|_{\widetilde A} + \|\bZ u_i- \bP u_i \|_{\widetilde A} =\|u_i-\bZ u_i\|_{\widetilde A} + \|\bP (u_i-  \bG u_i-\alpha) \|_{\widetilde A}. 
$$
Applying Lemma \ref{Plesssimon} for the last term gives
$$
\|\bP (u_i-  \bG u_i-\alpha) \|_{\widetilde A}\leq  \sqrt{B} \|u_i-  \bG u_i-\alpha\|_{\widetilde M},
$$
and as a consequence
$$
\|u_i-\bP u_i\|_{\widetilde A} \leq\|u_i-\bZ u_i\|_{\widetilde A} + \sqrt{B}\|u_i-  \bG u_i-\alpha\|_{\widetilde M}.
$$
\end{proof}

%WWWWWWWWWWWWWWWWWWWWWWWWWWWWWWWWWWWWWWWWWWWWWWWWWWWWWWWWWWWWWWWW
%WWWWWWWWWWWWWWWWWWWWWWWWWWWWWWWWWWWWWWWWWWWWWWWWWWWWWWWWWWWWWWWW
%WWWWWWWWWWWWWWWWWWWWWWWWWWWWWWWWWWWWWWWWWWWWWWWWWWWWWWWWWWWWWWWW
%WWWWWWWWWWWWWWWWWWWWWWWWWWWWWWWWWWWWWWWWWWWWWWWWWWWWWWWWWWWWWWWW
%WWWWWWWWWWWWWWWWWWWWWWWWWWWWWWWWWWWWWWWWWWWWWWWWWWWWWWWWWWWWWWWW
%WWWWWWWWWWWWWWWWWWWWWWWWWWWWWWWWWWWWWWWWWWWWWWWWWWWWWWWWWWWWWWWW
%WWWWWWWWWWWWWWWWWWWWWWWWWWWWWWWWWWWWWWWWWWWWWWWWWWWWWWWWWWWWWWWW

\section{Numerical Results for Eigenfunctions} \label{sec5}
Let $\gamma$ be the unit sphere in $\mathbb{R}^3$. The eigenfunctions of the Laplace-Beltrami operator are then the spherical harmonics.  The eigenvalues are given by $\ell(\ell+1)$, $\ell=1,2,3...$, with multiplicity $2 \ell+1$.  Computations were performed on a sequence of uniformly refined quadrilateral meshes using deal.ii \cite{BHK:07}; our proofs extend to this situation with modest modifications.  
When comparing norms of errors we took the first spherical harmonic for each eigenvalue $\ell(\ell+1)$ as the exact solution and then projected this function onto the corresponding discrete invariant space having dimension $2\ell+1$.  

\subsection{Eigenfunction error rates}
We calculated the eigenfunction error $\|u_1-\bP u_1\|_{\tilde{M}}$ and $\|u_1-\bP u_1\|_{\tilde{A}}$ for the lowest spherical harmonic corresponding to $\lambda_1=2$. From Theorem \ref{t:L2} and the results of \cite{D09}, we expect 
\begin{equation}
\|u_1-\bP u_1\|_{\tilde{M}}\lesssim C(\lambda)(h^{r+1}+h^{k+1}).
\end{equation}
From Proposition \ref{prop:equiv} and Theorem \ref{t:energy}, we expect 
\begin{equation}
\|u_1-\bP u_1\|_{\tilde{A}}\lesssim C(\lambda)(h^{r} + h^{k+1}).
\label{ProjError}
\end{equation}
We postpone discussion of dependence of the constants on spectral properties to Section~\ref{ss:numeric_constant}. When $r=1$ and $k=2$,the $L_2$ error is dominated by the PDE approximation (Figure \ref{fig2}), $h^{k+1} = h^3 \lesssim h^2=h^{r+1}$. When $r=3$ and $k=1$ we see the $L_2$ error is dominated by the geometric approximation (Figure 2), $h^{r+1} =h^4 \lesssim h^2 =h^{k+1}$. This illustrate the sharpness of our theory with respect to the approximation degrees. The energy error behavior reported in Figure \ref{fig2} similarly indicates that \eqref{ProjError} is sharp.  

\setlength{\unitlength}{.75cm}
\begin{figure}[h]
\centering
\includegraphics[scale=.42]{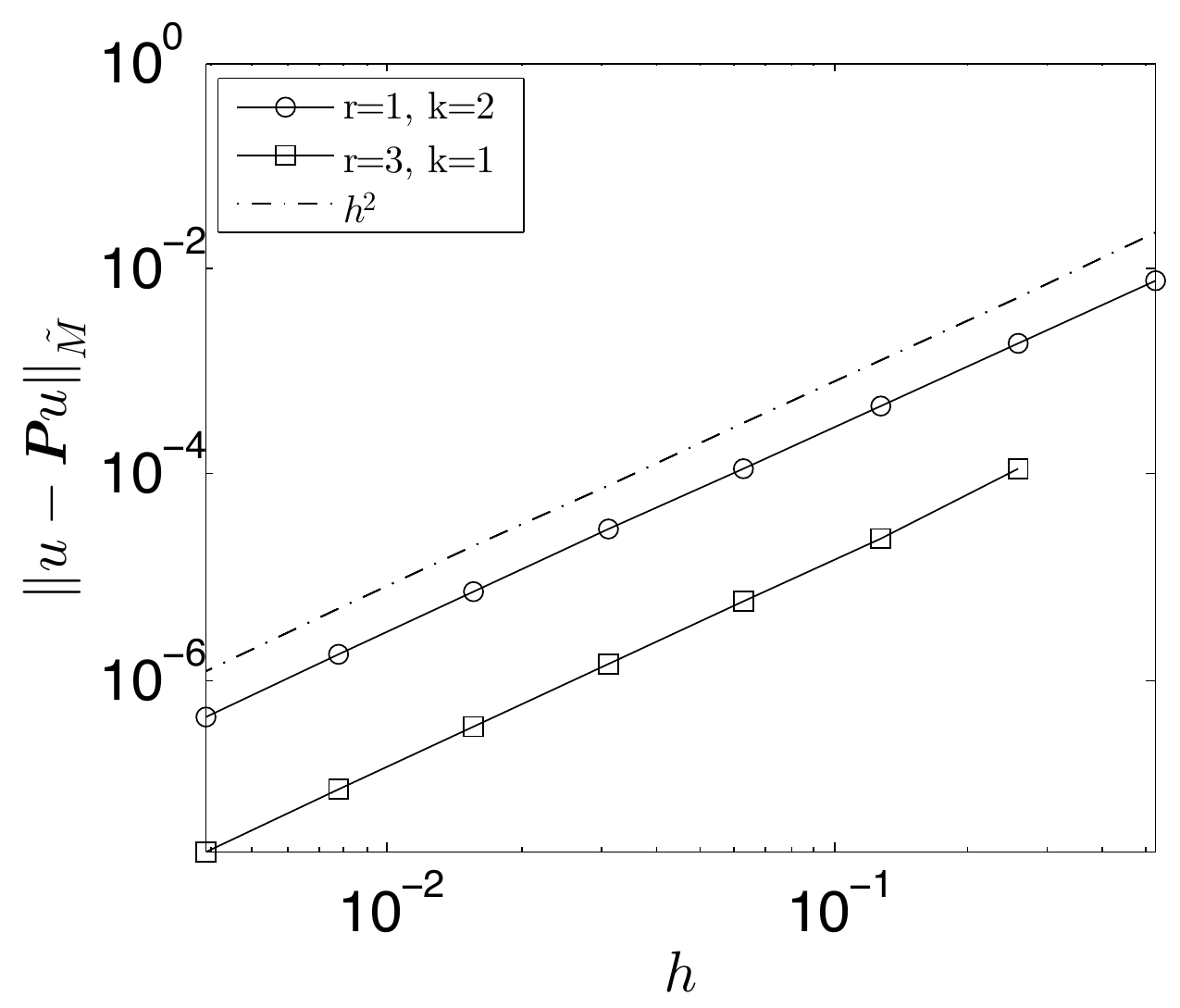}
\includegraphics[scale=.42]{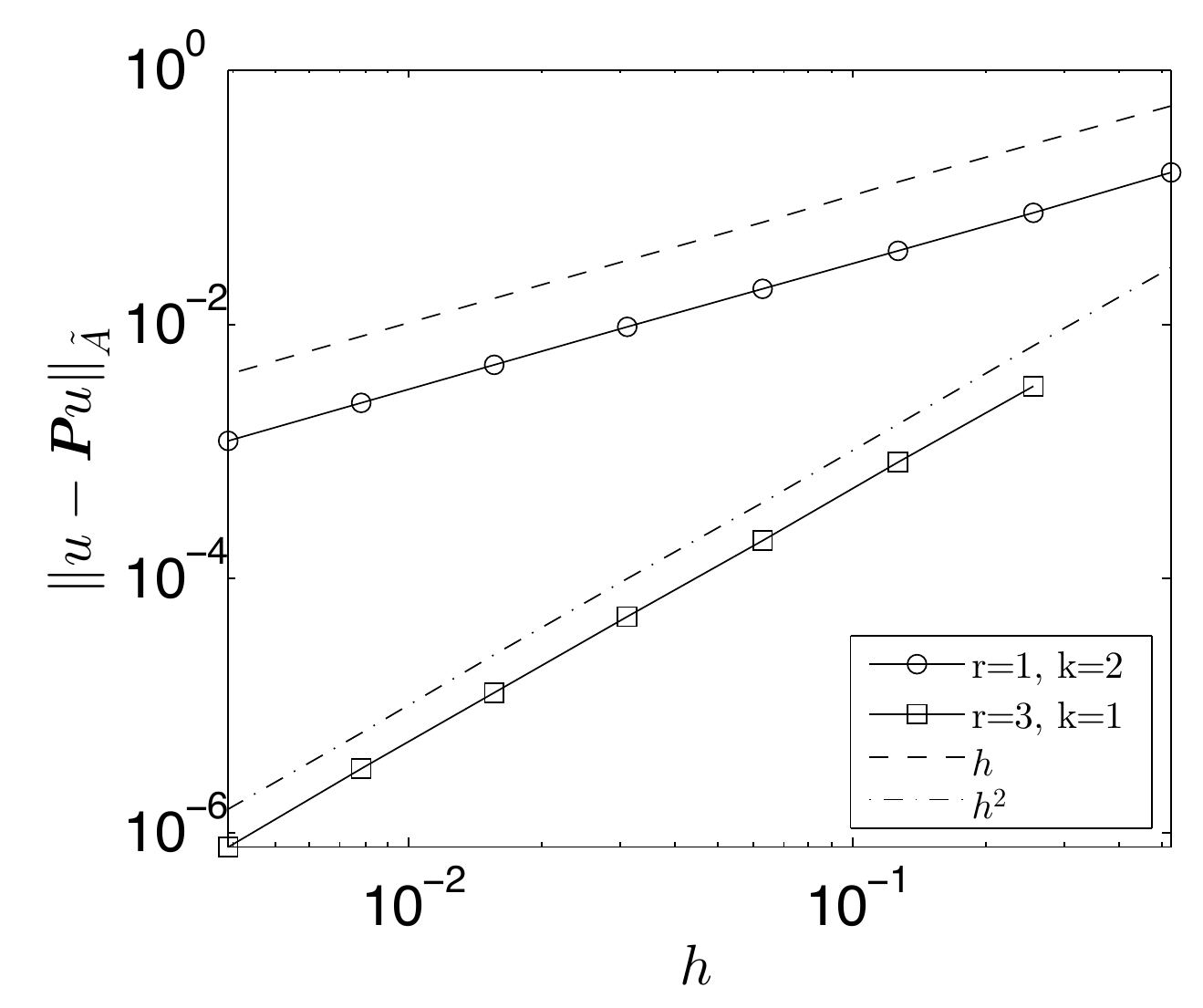}
\caption{Convergence rates of the approximate invariant eigenspace corresponding to the first eigenvalue on the sphere:  $L_2$ errors (left) and energy errors (right).}
\label{fig2}
\end{figure}

\subsection{Numerical evaluation of constants}\label{ss:numeric_constant}
In the left plot of Figure \ref{fig3} we plot $\frac{\|u-\bP u u\|_{\ak}}{\sqrt{\lambda}(1+\mu(J))h^{k+1}}$ vs. $h$ for $r=3$ and $k=1$ to evaluate the quality of our constant in Theorem \ref{t:energy}.  Here the Galerkin error is $O(h^4)$ and the geometric error $O(h^2)$, so the geometric error dominates.  Consider the eigenvalues $\lambda=\ell (\ell+1)$, $\ell=1,..., 10$ and corresponding spherical harmonics.  We chose two different exact spherical harmonics for $\ell=10$ to determine whether the choice of harmonic would affect the computation.  In the left plot of Figure \ref{fig3}, we see that the ratio $\frac{\|u-\bP u\|_{\ak}}{\sqrt{\lambda}(1+\mu(J))h^{k+1}}$ decreases moderately as $\lambda$ increases, indicating that the constant in Theorem \ref{t:energy} may not be sharp.  We thus also plotted $\frac{\|u-\bP u\|_{\ak}}{\sqrt{\lambda}(2+\sqrt{\mu(J)})h^{k+1}}$ and found this quantity to be more stable as $\lambda$ increases (see the right plot of Figure \ref{fig3}).  Thus it is possible that the dependence of the constant in front of the geometric error term in Theorem \ref{t:energy} is not sharp with respect to its dependence on $\mu(J)$. Our method of proof does not seem to provide a pathway to proving a sharper dependence, however, and our numerical experiments do confirm that the constant in front of the geometric error depends on spectral properties.  

In Figure \ref{fig4} we similarly test the sharpness of the geometric constant in the eigenvalue error estimate \eqref{eig_est} by plotting $\frac{|\lambda-\Lambda|}{\lambda h^2}$. This quantity is very stable as $\lambda$ increases, thus verifying the sharpness of the estimate as well as the correctness of the order, $O(h^{k+1})$ for $k=1$. In Section 7 we observe that for $k\geq 2$ the geometric error is between $h^{k+1}$ and $h^{2k}$. We delay giving numerical details until laying a theoretical foundation for explaining these superconvergence results. 

\setlength{\unitlength}{.75cm}
\begin{figure}[h]
\centering
\includegraphics[scale=.34]{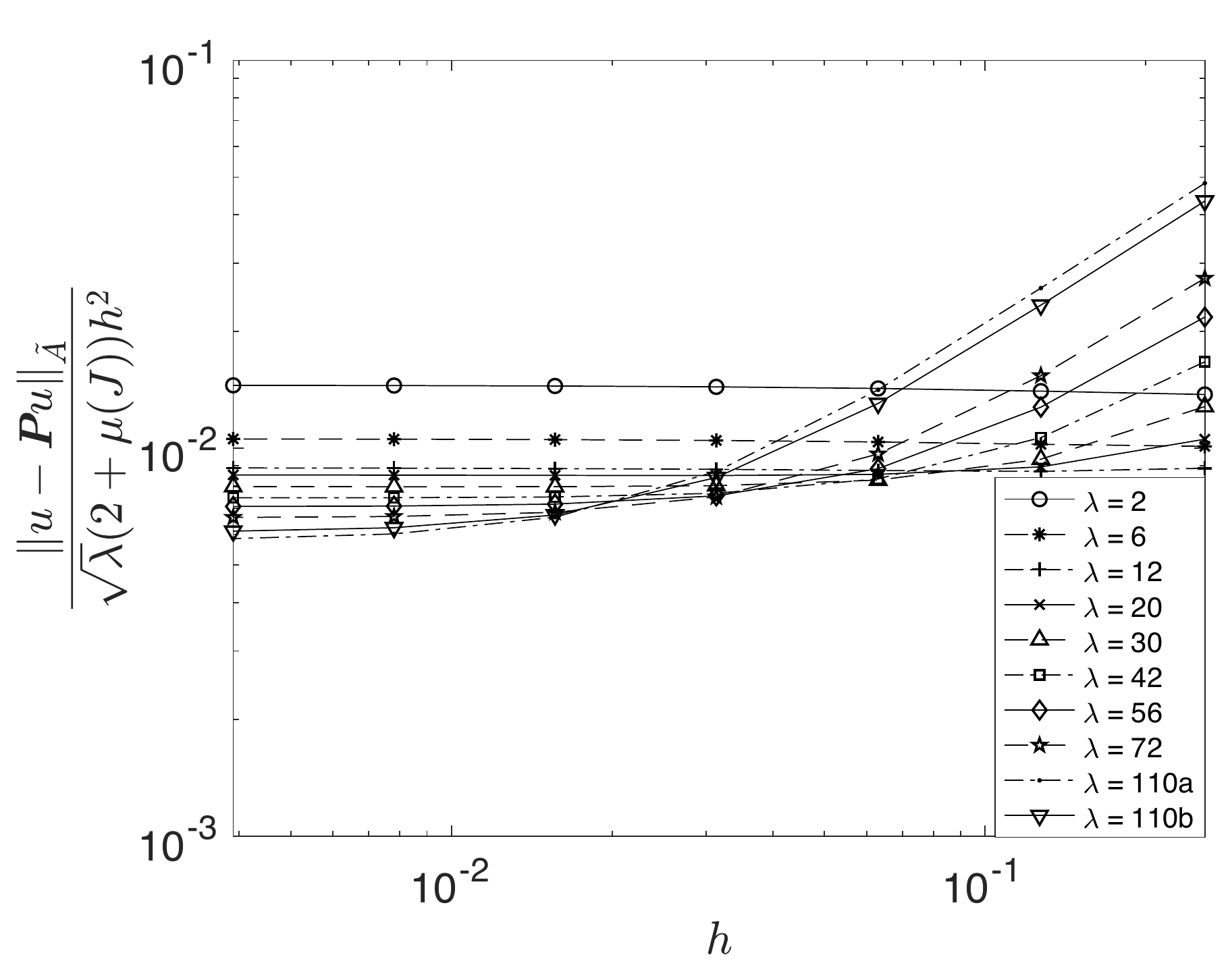}
\includegraphics[scale=.34]{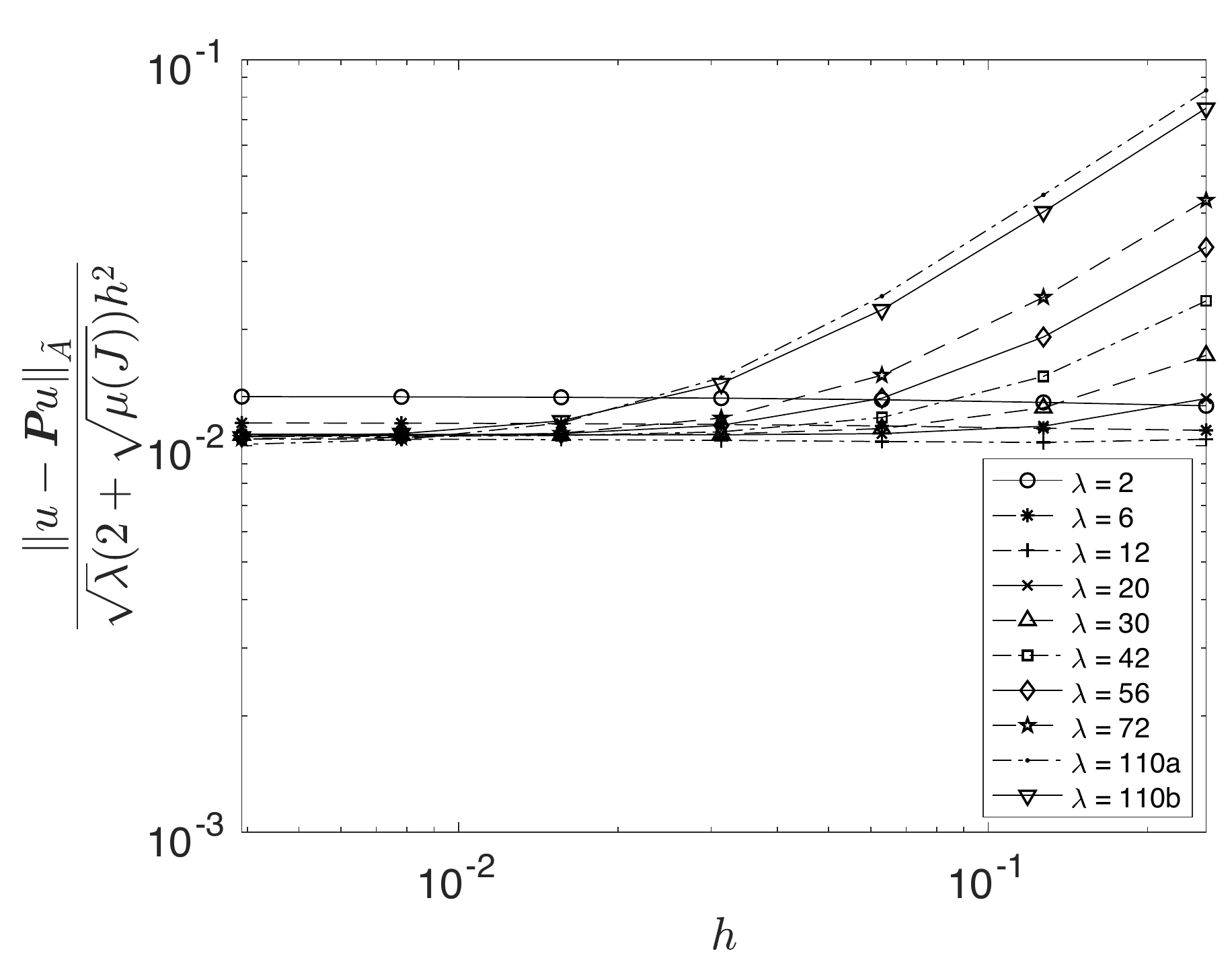}
\caption{Dependence of geometric portion in energy errors on spectral constants:  Theoretically established constant $\frac{\|u-\Ek u\|_{\ak}}{\sqrt{\lambda}(1+\mu(J))h^{k+1}}$ (left) and conjectured constant $\frac{\|u-\Ek u\|_{\ak}}{\sqrt{\lambda}(2+\sqrt{\mu(J)})h^{k+1}}$ (right).}
\label{fig3}
\end{figure}

\setlength{\unitlength}{.75cm}
\begin{figure}[h]
\centering
\includegraphics[scale=.36]{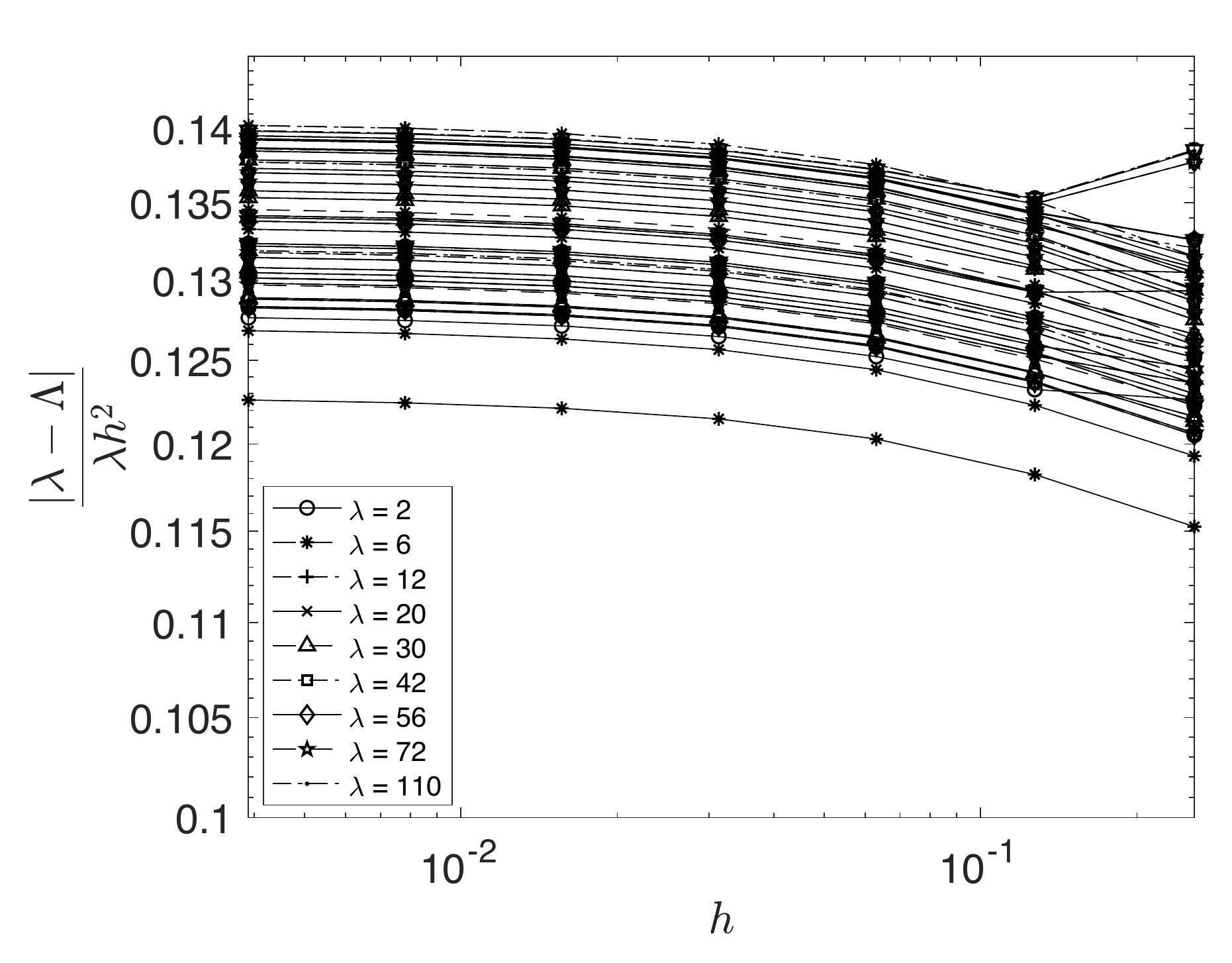}
\caption{Dependence of geometric portion of eigenvalue errors on spectral constants, $k=1$:  Theoretically established constant $\frac{|\lambda-\Lambda|}{\lambda h^2}$ for eigenvalues $\ell(\ell+1)$, $\ell=1,..., 10$.}
\label{fig4}
\end{figure}

%WWWWWWWWWWWWWWWWWWWWWWWWWWWWWWWWWWWWWWWWWWWWWWWWWWWWWWWWWWWWWWWW
%WWWWWWWWWWWWWWWWWWWWWWWWWWWWWWWWWWWWWWWWWWWWWWWWWWWWWWWWWWWWWWWW
%WWWWWWWWWWWWWWWWWWWWWWWWWWWWWWWWWWWWWWWWWWWWWWWWWWWWWWWWWWWWWWWW
%WWWWWWWWWWWWWWWWWWWWWWWWWWWWWWWWWWWWWWWWWWWWWWWWWWWWWWWWWWWWWWWW
%WWWWWWWWWWWWWWWWWWWWWWWWWWWWWWWWWWWWWWWWWWWWWWWWWWWWWWWWWWWWWWWW
%WWWWWWWWWWWWWWWWWWWWWWWWWWWWWWWWWWWWWWWWWWWWWWWWWWWWWWWWWWWWWWWW
%WWWWWWWWWWWWWWWWWWWWWWWWWWWWWWWWWWWWWWWWWWWWWWWWWWWWWWWWWWWWWWWW

\section{Superconvergence of Eigenvalues} \label{sec6}
In this section we analyze the geometric error estimates (2.16) and (2.17) from the viewpoint of numerical integration. Our approach is not cluster robust, but allows us to analyze superconvergence effects and leads to a characterization of the relationship between the choice of interpolation points in the construction of $\Gamma$ and the convergence rate for the eigenvalues. We show that we may obtain geometric errors of order $O(h^\ell)$ for $k+1\leq \ell \leq 2k$ by choosing interpolation points in the construction of $\Gamma$ that correspond to a quadrature scheme of order $\ell$. Because these superconvergence effects require a more subtle analysis, we do not trace the dependence of constants on spectral properties in this section and are only interested in orders of convergence. We denote the untracked spectrally dependent constant by $C_\lambda$, which may change values throughout the calculations.

We first state a result similar to \cite[Theorem 5.1]{BO90}, where effects of numerical quadrature on eigenvalue convergence were analyzed.  Let $\lambda_j$ be an eigenvalue of \eqref{weig} with multiplicity $N$.  Let $\W$ and $\W_\#$ be the spans of the eigenfunctions of $\lambda_j$ and the $N$ FEM eigenfunctions associated with the approximating eigenvalues of $\lambda_j$.
\begin{lemma}{Eigenvalue Bound.} \label{MacroBound}	
 Let $\bP_{\lambda_j}$ be the projection onto $\W$ using the $L_2$ inner product $m(\cdot,\cdot)$. Let $U_j$ be an eigenfunction in $\W_\#$ such that $\|U_j\|_m=1$ and $A(U_j,U_j) = \Lambda_jM(U_j,U_j)$. Then
\begin{align}
\label{eigbound}
\begin{aligned}
|& \lambda_j  - \Lambda_j| = \left|\frac{a(\bP_{\lambda_j}U_j,\bP_{\lambda_j}U_j)}{m(\bP_{\lambda_j}U_j,\bP_{\lambda_j}U_j)}- \frac{\widetilde{A}(U_j,U_j)}{\widetilde{M}(U_j,U_j)}\right|  \leq \|\bP_{\lambda_j}U_j-U_j\|_a^2  \\ & + \lambda_j\|\bP_{\lambda_j}U_j-U_j\|_m^2
+ \Lambda_j|m(U_j,U_j)-\widetilde{M}(U_j,U_j)| + |\widetilde{A}(U_j,U_j) - a(U_j,U_j)|.
\end{aligned}	
\end{align}
\end{lemma}
\begin{proof}
	Since $a(\bP_{\lambda_j}U_j,U_j) = \lambda_j m(\bP_{\lambda_j}U_j,U_j)$ and $\|\bP_{\lambda_j}U_j\|_a^2 = \lambda_j\|\bP_{\lambda_j}U_j\|_m^2$,
	$$
	\|\bP_{\lambda_j}U_j-U_j\|_a^2  - \lambda_j \|\bP_{\lambda_j}U_j-U_j\|_m^2 = \|\bP_{\lambda_j}U_j\|_a^2 + \|U_j\|_a^2 - 2a(\bP_{\lambda_j}U_j,U_j)
	$$
	$$
	 - \lambda_j\|\bP_{\lambda_j}U_j\|_m^2
	 + 2\lambda_jm(\bP_{\lambda_j}U_j,U_j) - \lambda_j\|U_j\|_m^2
    =
	a(U_j,U_j) - \lambda_j\|U_j\|_m^2.
	$$
	Noting the assumption that $\|U_j\|_m = 1$, we get
	\begin{equation}
		- \lambda_j = \|\bP_{\lambda_j}U_j-U_j\|_a^2 - \lambda_j\|\bP_{\lambda_j}U_j-U_j\|_m^2 - a(U_j,U_j). 
		\label{step1}
	\end{equation}
	Because $\widetilde{A}(U_j,U_j)-\Lambda_j\widetilde{M}(U_j,U_j)=0$ we get 
	
	$$
	- \lambda_j = \|\bP_{\lambda_j}U_j-U_j\|_a^2 - \lambda_j\|\bP_{\lambda_j}U_j-U_j\|_m^2  +[\widetilde{A}(U_j,U_j)- a(U_j,U_j)]  - \Lambda_j\widetilde{M}(U_j,U_j).
	$$
	Adding $\Lambda_j = \Lambda_jm(U_j,U_j)$ to both sides and taking absolute values gives the result.
\end{proof}

We now give a series of results bounding the terms on the right hand side of \eqref{eigbound}.  Recall that $\bP$ denotes the $\widetilde M$ projection onto $\W_\#$.  %See \cite[Lemma 5.1]{Gallistl} for a proof of the following lemma.

\begin{lemma}
For $h$ small enough, $\{\bP u:u\in\W\}$ forms a basis for $\text{span}\{U:U\in\W_\#\}$. Moreover, for any $U\in \W_\#$ with $\|U\|_m=1$, %the coefficients of the representation $U = \sum_{i=1}^N\alpha_i\bP u_i$ are bounded, i.e.
\begin{equation}\label{coeffbound}
\sum_{i=1}^N|\alpha_i|^2\leq C(N).
\end{equation}
\end{lemma}	
\begin{proof}
The proof follows the same steps given in the proof of \cite[Lemma 5.1]{Gallistl}.
\end{proof}	
\begin{lemma} \label{P-H}	
Let $h$ be small enough that $\{\bP u:u\in\W\}$ forms a basis for $\text{span}\{U:U\in\W_\#\}$. Let $\{u_i\}_{i=1}^N$ be an orthonormal basis for $\W$ with respect to $m(\cdot, \cdot)$. Then	
	\begin{align}
		\|U - \bP_{\lambda_j}U\|_a &\leq C_\lambda \max_{i=1,...,N}\|u_i-\bP u_i\|_a  \lesssim h^r + h^{k+1},
		\label{abound}
\\		\|U - \bP_{\lambda_j}U\|_m & \leq C_\lambda \max_{i=1,...,N}\|u_i-\bP u_i\|_m \lesssim h^{r+1} + h^{k+1}
		\label{mmbound}
	\end{align}
for any $u\in\W$ and $U\in\W_\#$.
\end{lemma}
\begin{proof}
Recall that $N =\dim(\W)$. Since $U\in \text{span}\{\bP u:u\in\W\}$, there holds $U = \sum_{i=1}^N\alpha_i\bP u_i$ with the coefficients satisfying \eqref{coeffbound}.  Thus
	$$
	\bP_{\lambda_j}U - U = \sum_{k=1}^Nm(\sum_{i=1}^N\alpha_i\bP u_i,u_k)u_k - \sum_{i=1}^N\alpha_i\bP u_i.
	$$
	Adding $-\sum_{i=1}^N\alpha_i m(u_i,u_i)u_i + \sum_{i=1}^N\alpha_iu_i = 0$ and using $m(u_i,u_k) = 0$, $i\neq k$, yields
\begin{equation}
	\bP_{\lambda_j}U - U=\sum_{i=1}^N\alpha_i\left(\sum_{k=1}^Nm(\bP u_i-u_i,u_k)u_k +(u_i- \bP u_i) \right).
	\label{keyform}
\end{equation}
Using $m(\bP u_i-u_i,u_k) = \frac{1}{\lambda_j}a(\bP u_i-u_i,u_k)$, noting \eqref{coeffbound} and applying $\|\cdot\|_a$
to both sides of \eqref{keyform} yields the first inequality in \eqref{abound}, while applying $\|\cdot\|_m$ to both sides of \eqref{keyform} yields similarly the first inequality in \eqref{mmbound}. The second inequality in \eqref{abound} follows from Proposition \ref{prop:equiv} and \eqref{D-abound}.

To obtain the second inequality in \eqref{mmbound}, we first use \eqref{l2_bound} and $\|\cdot\|_m\simeq\|\cdot\|_{\tilde{M}}$: 
\begin{equation}\label{mbound}
\begin{aligned}
\|u_k-\bP u_k\|_m & \lesssim \|u_k-\bG u_k\|_m \\ & \leq \|u_k-\bG u_k - m(u_k-\bG u_k,1)\|_m 
+ \|m(u_k-\bG u_k,1)\|_m.
\end{aligned}
\end{equation}
Since $m(u_k,1)=\widetilde{M}(u_k,1)=0$, we have from \eqref{m-mk} that 
\begin{align*}
\|m(u_k-\bG u_k,1)\|_m &= \|m(\bG u_k,1)\|_m
\\
&=\sqrt{|\gamma|}|m(\bG u_k,1) - \widetilde{M}(\bG u_k,1)|\leq |\gamma|\|\bG u_k\|_{\widetilde{M}}h^{k+1}.
\end{align*}
Also, $\|\bG u_k\|_{\widetilde{M}} \lesssim \|\bG u_k\|_{\widetilde{A}}\lesssim \| u_k\|_{a}\lesssim C_\lambda$.
Bounding the first term on the right hand side of \eqref{mbound} using \eqref{D-mbound} completes the proof.						
\end{proof}

\begin{lemma}
Let $v\in H_\#^1(\gamma)$, let $d(x)$ be the signed distance function for $\gamma$, let $\psi(x)$ be the closest point projection onto $\gamma$, let $\nu$ be the normal vector of $\gamma$, let $\N$ be the normal vector of $\Gamma$, and $\{\mathbf{e}_i\}_{i=1}^n$ be the eigenvectors of the Hessian, $\mathbf{H}$, of $\gamma$, then 	
\begin{align}
&\begin{aligned}
|a(v,v)& -\widetilde{A}(v,v)| \leq
\left|\int_{\Gamma}d(x)\mathcal{H}\left[\nabla_{\Gamma}v\right]^T\nabla_\Gamma v d\Sigma\right|\\
&+2\left|\int_{\Gamma}d(x)\left(\sum_{i=1}^n\kappa_i(\mathbf{\psi}(x))\left[\nabla_{\Gamma}v\right]^T\left[ \mathbf{e}_i\otimes\mathbf{e}_i 
\right]\nabla_\Gamma v \right)d\Sigma\right| + O(h^{2k}),
\label{StiffBound}	
\end{aligned}
\\
&\begin{aligned}
\left|m(v,v)-\widetilde{M}(v,v)\right|&\leq\left|\int_{\Gamma}v^2d(x)\mathcal{H}d\Sigma \right|+ O(h^{2k}).
\label{massbound}
\end{aligned}
\end{align}
Here $\mathcal{H} = \sum_{i=1}^n\kappa_i(\mathbf{\psi}(x))$ is the scaled mean curvature of $\gamma$.
\end{lemma}	

\begin{proof}
	We shall need the two identities from \cite{DD07}:
	\begin{align}
		\nabla_{\gamma}v(x) &= [(\mathbf{I} - d\mathbf{H})(x)]^{-1}\left[ \mathbf{I} - \frac{\mathbf{N}\otimes\mathbf{\nu}}{\mathbf{N}\cdot\mathbf{\nu}} \right]\nabla_{\Gamma}v,
		\label{GradId}
\\
		d\sigma &= \mathbf{\nu}\cdot\mathbf{N}\left[\prod_{i=1}^n\left(1-d(x)\frac{\kappa_i(\mathbf{\psi}(x))}{1 + d(x)\kappa_i(\mathbf{\psi}(x))}\right)\right]d\Sigma := Qd\Sigma.
		\label{Jacob}
	\end{align}
	We note that since $|1-\nu\cdot \mathbf{N}| = \frac{1}{2}|\nu-\mathbf{N}|^2\lesssim h^{2k}$ and $\|d\|_{L_\infty(\Gamma)}\lesssim h^{k+1}$,
	\begin{equation}
	Q = (1-d\mathcal{H}) + O(h^{2k}).
	\label{appJacob}
	\end{equation}
Using \eqref{GradId} and \eqref{appJacob} we then have
\begin{align}
\begin{aligned}
	|a(v,v)& -\widetilde{A}(v,v)| = \left|\int_{\gamma}[\nabla_\gamma v]^T\nabla_\gamma v d\sigma - \int_{\Gamma} [\nabla_\Gamma v]^T\nabla_\Gamma v d\Sigma\right|\\
	&\leq \bigg|\int_{\Gamma}\left[\nabla_{\Gamma}v\right]^T\left[ \mathbf{I} - \frac{\mathbf{\nu}\otimes\mathbf{N}}{\mathbf{N}\cdot\mathbf{\nu}}\right][[(\mathbf{I} - d\mathbf{H})(x)]^{-1}]^T[(\mathbf{I} - d\mathbf{H})(x)]^{-1}\\
	&\times\left[ \mathbf{I} - \frac{\mathbf{N}\otimes\mathbf{\nu}}{\mathbf{N}\cdot\mathbf{\nu}} \right]
	\nabla_{\Gamma}v \left[1-d(x)\mathcal{H}\right] 
	- [\nabla_\Gamma v]^T \nabla_\Gamma vd\Sigma\bigg|
	+O(h^{2k}).
\end{aligned}
\label{One}	
\end{align}
	Expanding the Hessian $\mathbf{H}$ as on page 425 of \cite{DD07}, we obtain:
	$$
	[(\mathbf{I} - d\mathbf{H})(x)]^{-1} = \mathbf{\nu}\otimes\mathbf{\nu} + \sum_{i=1}^n[1 + d(x)\kappa_i(\mathbf{\psi}(x))]\mathbf{e}_i\otimes\mathbf{e}_i
	= \mathbf{I} + \sum_{i=1}^nd(x)\kappa_i(\mathbf{\psi}(x))\mathbf{e}_i\otimes\mathbf{e}_i.
	$$
Using $\mathbf{e}_i\perp\mathbf{\nu}$ and $\mathbf{e}_i\perp\mathbf{e}_j$, $1\leq i,j\leq n$, yields
\begin{align*}
	[[(\mathbf{I} - d\mathbf{H})(x)]^{-1}]^T[(\mathbf{I} - d\mathbf{H})(x)]^{-1} =\mathbf{I} + 2\sum_{i=1}^nd(x)\kappa_i(\mathbf{\psi}(x))\mathbf{e}_i\otimes\mathbf{e}_i + O(h^{2k+2}).
\end{align*}
Combining the above and carrying out a short calculation yields
\begin{align*}
	&\left[ \mathbf{I} - \frac{\mathbf{\nu}\otimes\mathbf{N}}{\mathbf{N}\cdot\mathbf{\nu}}\right][[(\mathbf{I} - d\mathbf{H})(x)]^{-1}]^T[(\mathbf{I} - d\mathbf{H})(x)]^{-1}\left[\mathbf{I} - \frac{\mathbf{N}\otimes\mathbf{\nu}}{\mathbf{N}\cdot\mathbf{\nu}} \right]\\
	&=
	\left[ \mathbf{I} - \frac{\mathbf{\nu}\otimes\mathbf{N}}{\mathbf{N}\cdot\mathbf{\nu}}\right][\mathbf{I} + 2\sum_{i=1}^nd(x)\kappa_i(\mathbf{\psi}(x))\mathbf{e}_i\otimes\mathbf{e}_i]\left[ \mathbf{I} - \frac{\mathbf{N}\otimes\mathbf{\nu}}{\mathbf{N}\cdot\mathbf{\nu}} \right] + O(h^{2k})\\
	&=
	 \mathbf{I} - \frac{\mathbf{\nu}\otimes\mathbf{N}}{\mathbf{N}\cdot\mathbf{\nu}} - \frac{\mathbf{N}\otimes\mathbf{\nu}}{\mathbf{N}\cdot\mathbf{\nu}} + \frac{\mathbf{\nu}\otimes\mathbf{\nu}}{(\mathbf{N}\cdot\mathbf{\nu})^2}\\
	&+ 2\sum_{i=1}^nd(x)\kappa_i(\mathbf{\psi}(x))\left[ \mathbf{e}_i\otimes\mathbf{e}_i - \frac{\mathbf{N}\cdot \mathbf{e}_i}{\mathbf{N}\cdot\mathbf{\nu}}\left(\mathbf{\nu}\otimes\mathbf{e}_i
	+\mathbf{e}_i\otimes\mathbf{\nu}\right)
	+ \left(\frac{\mathbf{N}\cdot \mathbf{e}_i}{\mathbf{N}\cdot\mathbf{\nu}}\right)^2\mathbf{\nu}\otimes\mathbf{\nu}\right]
	\\
	&+ O(h^{2k}).
\end{align*}
Let $P_\Gamma:=\mathbf{I}-\mathbf{N}\otimes\mathbf{N}$. Then
\begin{align*}
 \mathbf{I} - \frac{\mathbf{\nu}\otimes\mathbf{N}}{\mathbf{N}\cdot\mathbf{\nu}} - \frac{\mathbf{N}\otimes\mathbf{\nu}}{\mathbf{N}\cdot\mathbf{\nu}} + \frac{\mathbf{\nu}\otimes\mathbf{\nu}}{(\mathbf{N}\cdot\mathbf{\nu})^2}
 &=
 P_\Gamma + \left(\mathbf{N} - \frac{\mathbf{\nu}}{\mathbf{N}\cdot\mathbf{\nu}}\right)\otimes\left(\mathbf{N} - \frac{\mathbf{\nu}}{\mathbf{N}\cdot\mathbf{\nu}}\right)
\\&=
 P_\Gamma + O(h^{2k}).
\end{align*}
We know $\|\mathbf{N}-\mathbf{\nu}\|_\infty \lesssim h^{k}$, so $\mathbf{N}\cdot \mathbf{e}_i = O(h^{k})$ which means all terms containing $d(x)\mathbf{N}\cdot \mathbf{e}_i$ are of order $h^{2k+1}$. Therefore we have
\begin{align}
\begin{aligned}
&\left[ \mathbf{I} - \frac{\mathbf{\nu}\otimes\mathbf{N}}{\mathbf{N}\cdot\mathbf{\nu}}\right][[(\mathbf{I} - d\mathbf{H})(x)]^{-1}]^T[(\mathbf{I} - d\mathbf{H})(x)]^{-1}\left[ \mathbf{I} - \frac{\mathbf{N}\otimes\mathbf{\nu}}{\mathbf{N}\cdot\mathbf{\nu}} \right]\\
&=P_\Gamma
+ 2\sum_{i=1}^nd(x)\kappa_i(\mathbf{\psi}(x))\left[ \mathbf{e}_i\otimes\mathbf{e}_i\right] 
+ O(h^{2k}).
\end{aligned}
\label{Guts}
\end{align}
Multiplying equations \eqref{Guts} and \eqref{appJacob} gives
\begin{align*}
	&\left[ \mathbf{I} - \frac{\mathbf{\nu}\otimes\mathbf{N}}{\mathbf{N}\cdot\mathbf{\nu}}\right][[(\mathbf{I} - d\mathbf{H})(x)]^{-1}]^T[(\mathbf{I} - d\mathbf{H})(x)]^{-1}\left[ \mathbf{I} - \frac{\mathbf{N}\otimes\mathbf{\nu}}{\mathbf{N}\cdot\mathbf{\nu}} \right]
	Q\\
	&=P_\Gamma(1-d(x)\mathcal{H})
	+2\sum_{i=1}^nd(x)\kappa_i(\mathbf{\psi}(x))\left[ \mathbf{e}_i\otimes\mathbf{e}_i 
	\right] + O(h^{2k}).
\end{align*}
Inserting the above into \eqref{One} and noting that $P_\Gamma\nabla_\Gamma v = \nabla_\Gamma v$ yields 
\begin{align*}
	|a(v,v)-\widetilde{A}(v,v)|
	&\leq \left|\int_{\Gamma}d(x)\mathcal{H}\left|\nabla_{\Gamma}v\right|^2d\Sigma\right|\\
	&+2\left|\int_{\Gamma}\left(\sum_{i=1}^nd(x)\kappa_i(\mathbf{\psi}(x))\left[\nabla_{\Gamma}v\right]^T\left[ \mathbf{e}_i\otimes\mathbf{e}_i 
	\right]\nabla_\Gamma v\right) d\Sigma\right| + O(h^{2k}).
\end{align*}
This is \eqref{StiffBound}.	The proof of \eqref{massbound} follows directly from \eqref{appJacob}.
\end{proof}	

We next define a quadrature rule on the reference element:
$$
\int_{\hat{T}}\hat{\varphi}(\hat{x})d\hat{\Sigma}\approx \sum_{i=1}^L \hat{w}_i\hat{\varphi}(\hat{q}_i),
$$
where $\{\hat{w}_j\}_{j=1}^L$ are weights and $\{\hat{q}_j\}_{j=1}^L$ is a set of quadrature points. Recall the definition \eqref{effT} of $\bF_T:\hat{T} \rightarrow T$.  %If $\mathbf{F}_{\overline{T}}:\hat{T}\to \overline{T}\in \overline{\Gamma}$, then we 
The mapped rule on a physical element $T \subset \Gamma$ is
$$
\int_{T}\varphi(x)d\Sigma\approx \sum_{i=1}^L w_i\varphi(q_i),
$$
where $w_i = Q_{\bF_T}(\hat{q}_i)\hat{w}_i$,  $Q_{\bF_T} = \sqrt{\det(J^TJ)}$ with $J$ the Jacobian matrix of $\bF_T$, and $q_i = \bF_T(\hat{q}_i)$.  The quadrature errors on the unit and physical elements are
\begin{equation}
E_{\hat{T}}(\varphi) := \int_{\hat{T}}\hat{\varphi}(\hat{x})d\hat{\Sigma} - \sum_{i=1}^L \hat{w}_i\hat{\varphi}(\hat{q}_i), \qquad
E_T(\varphi) := \int_{T}\varphi(x)d\Sigma - \sum_{i=1}^L w_i\varphi(q_i).
\label{QuadRule}
\end{equation}

We say that a mapping $\bF_T$ is {\it regular} if $|\bF_T|_{W^{i,\infty}(\hat{T})}\leq h^i$, $0 \le i \le k$. This is implied by assumption \eqref{e:lift_estim}.  Note also that $|\bF_T|_{W^{i,\infty}(\hat{T})}=0$, $i >k$.  
\begin{lemma}\label{QuadBound}
Suppose $E_{\hat{T}}(\hat{\chi})=0$ $\forall \hat{\chi}\in \mathbb{P}^{\ell-1}(\hat{T})$, $d\in W^{\ell,\infty}(T)$, and $\bF_T$ is a regular mapping. Then there is a constant $C$, independent of $T$, such that
\begin{equation}
	|E_T(d\varphi\psi)|\leq C\|d\|_{W^{\ell,\infty}(T)}h^{\ell}|\varphi|_{H^{\min\{r,\ell\}}(T)}|\psi|_{H^{\min\{r,\ell\}}(T)}, \quad \forall \hat{\varphi},\hat{\psi}\in\mathbb{P}^r(\hat{T}).
\end{equation}
\end{lemma}

\begin{proof}
We use standard steps from basic finite element theory \cite{Ciar02}.  For each $T$,
	\begin{equation}
		E_T(d\varphi\psi) =E_{\hat{T}}\left(d(\bF_T)Q_{\bF_T}\hat{\varphi}\hat{\psi}\right).
		\label{QuadErr}
	\end{equation}
	Since $E_{\hat{T}}(\hat{\chi})=0, \forall \hat{\chi}\in \mathbb{P}^{\ell-1}(\hat{T})$, it follows from the Bramble-Hilbert Lemma and \eqref{QuadRule} that 
	$$
	|E_{\hat{T}}(\hat{g})| = \inf_{\chi\in \mathbb{P}^{\ell-1}}|E_{\hat{T}}(\hat{g}-\chi)| \leq \inf_{\chi\in \mathbb{P}^{\ell-1}}\|\hat{g}-\chi\|_{L_\infty(\hat{T})}\leq \hat{C}|\hat{g}|_{W^{\ell,\infty}(\hat{T})}.
	$$
	Substituting $\hat{g}=d(\bF_T)Q_{\bF_T}\hat{\varphi}\hat{\psi}$, we thus have
	$$
	\left|E_{\hat{T}}\left(d(\bF_T)Q_{\bF_T}\hat{\varphi}\hat{\psi}\right)\right|\leq \hat{C}\left|d(\bF_T)Q_{\bF_T}\hat{\varphi}\hat{\psi}\right|_{W^{\ell,\infty}(\hat{T})}.
	$$
	We now apply equivalence of norms over finite dimensional spaces and scaling arguments noting that $D^\alpha\hat{\varphi}=D^\alpha\hat{\psi}=0$ for $|\alpha|>r$ to get 
	\begin{align*}
	\left|d(\bF_T)Q_{\bF_T}\hat{\varphi}\hat{\psi}\right|_{W^{\ell,\infty}(\hat{T})} &\leq \sum_{\substack{i,j=0\\ \ell-i-j\geq 0}}^{\min\{r,\ell\}}\left|d(\bF_T)Q_{\bF_T}\right|_{W^{\ell-i-j,\infty}(\hat{T})}|\hat{\varphi}|_{W^{i,\infty}(\hat{T})}|\hat{\psi}|_{W^{j,\infty}(\hat{T})}
\\
	&\lesssim \sum_{\substack{i,j=0\\ \ell-i-j\geq 0}}^{\min\{r,\ell\}}\left|d(\bF_T)Q_{\bF_T}\right|_{W^{\ell-i-j,\infty}(\hat{T})}|\hat{\varphi}|_{H^{i}(\hat{T})}|\hat{\psi}|_{H^{j}(\hat{T})}.
	\end{align*}
Through standard arguments we have
$$	
|\hat{\varphi}|_{H^{i}(\hat{T})}|\hat{\psi}|_{H^{j}(\hat{T})}
\lesssim 
h^{i+j}\|Q_{\bF_T^{-1}}\|_{L_{\infty}(T)}|\hat{\varphi}|_{H^{i}(T)}|\hat{\psi}|_{H^{j}(T)}.
$$
Noting that $\left|Q_{\bF_T}\right|_{W^{k,\infty}(\hat{T})}\lesssim h^{n+j}$ and $\|Q_{\bF_T^{-1}}\|_{L_{\infty}(T)}\lesssim h^{-n}$	along with 
	$$
	\left|d(\bF_T)Q_{\bF_T}\right|_{W^{\ell-i-j,\infty}(\hat{T})}
	\lesssim
	\sum_{k=0}^{\ell-i-j}\left|Q_{\bF_T}\right|_{W^{k,\infty}(\hat{T})}\left|d(\bF_T)\right|_{W^{\ell-i-j-k,\infty}(\hat{T})}
	$$
and	
	$$
	\left|d(\bF_T)\right|_{W^{\ell-i-j-k,\infty}(\hat{T})}
	\lesssim
	h^{\ell-i-j-k}\left\|d\right\|_{W^{\ell-i-j-k,\infty}(T)}
	$$
gives
$$
\left|d(\bF_T)Q_{\bF_T}\hat{\varphi}\hat{\psi}\right|_{W^{\ell,\infty}(\hat{T})}
\lesssim
h^{\ell}\|d\|_{W^{\ell,\infty}(\Omega)}\|\varphi\|_{H^{\min\{r,\ell\}}(T)}\|\psi\|_{H^{\min\{r,\ell\}}(T)},
$$	
which is the desired result.
\end{proof}

We now consider the effects of constructing $\Gamma$ by interpolating $\bpsi$.  
\begin{lemma}[Superconvergent Geometric Consistency] \label{bounds-H}	
Let $\text{QUAD}_{\hat{T}}$ be a degree $\ell-1$, $R$ point quadrature rule on the unit element with quadrature points $\{\hat{q}_i\}_{i=1}^R$, $V\in \Vh^r(\Gamma)$ be degree-$r$ function, and assume that $d(x)\mathcal{H}\in W^{\ell,\infty}(\mathcal{N})$. If the points $\{\bL(x^j)\}_{j=1}^{n_k}$ in \eqref{d:interpolated_lift} and $\{q_i\}_{i=1}^L$ coincide and in addition $\bL(x^j)=\bpsi(x^j)$, then
	\begin{align}
		|a(V,V)-\widetilde{A}(V,V)|&\leq
		h^\ell\left \|d(x)\mathcal{H}\right \|_{W_{\mathcal{T}}^{\ell,\infty}(\Gamma)}
		\left | V\right |_{H_{\mathcal{T}}^{\min\{r,\ell\}}(\Gamma)}^2+ O(h^{2k}),
	\label{stiff-H}	
\\
		|m(V,V)-\widetilde{M}(V,V)|&\lesssim h^{\ell}\left\|d(x)\mathcal{H}\right\|_{W_{\mathcal{T}}^{\ell,\infty}(\Gamma)}|V|_{H_{\mathcal{T}}^{\min\{r,\ell\}}(\Gamma)}^2 + O(h^{2k}).
	\label{mass-H}	
	\end{align}
\end{lemma}
Here a subscript $\mathcal{T}$ denotes a broken (elementwise) version of the given norm. 
\begin{proof}
	We prove \eqref{mass-H}.  \eqref{stiff-H} follows from similar arguments. Recalling \eqref{massbound} and partition the first integral based on the underlying mesh.
	$$
	\left|\int_{\Gamma}V^2d(x)\mathcal{H}d\Sigma\right|
	\leq
	\sum_{j=1}^{\# \space elements}\left|\int_{T_j}V^2d(x)\mathcal{H}d\Sigma\right|.
	$$
	Let $q$ be a quadrature point on $T_j$.  By assumption $\bL(q)=\bpsi(q)$, so $d(q) = 0$ and  
	$$
	\left|\int_{T_j}V^2d(x)\mathcal{H}d\Sigma\right| = \left|\int_{T_j}V^2d(x)\mathcal{H}d\Sigma - \text{QUAD}_{T_j}\left(V^2d(x)\mathcal{H}\right)\right|
	$$
	$$
	=E_{T_j}(d(x)\mathcal{H}V^2)
	\lesssim
	h^{\ell}\|d(x)\mathcal{H}\|_{W_{\mathcal{T}}^{\ell,\infty}(\Gamma)}|V|_{H_{\mathcal{T}}^{\min\{r,\ell\}}(T_j)}^2
	$$
	by Lemma \ref{QuadBound}. Summing over all of the elements yields \eqref{mass-H}.
\end{proof}

\begin{theorem}[Order of eigenvalue error] \label{Superconvergence}
If $\Gamma$ be constructed using interpolation points that correspond to a degree $\ell-1$ quadrature rule as in Lemma \ref{bounds-H}, then 
\begin{equation} \label{theorem_bound}
|\lambda_j-\Lambda_j|\lesssim h^{2r} + h^{2k} + h^\ell.
\end{equation}
\end{theorem}
\begin{proof}
Standard arguments (adding and subtracting an interpolant and applying inverse inequalities) yield $\|U\|_{H^k}\lesssim\|\mathbf{P}_{\lambda_j}U\|_{H^{k+1}}$.  Combining Lemma \ref{bounds-H} and Lemma \ref{P-H} into Lemma \ref{MacroBound} completes the proof.
\end{proof}

\begin{remark}
Our proofs carry over to the setting of quadrilateral elements with appropriate modification of the definition of regularity of the mapping $\bF_T$. If Gauss-Lobatto points are used on the faces of $\overline{\Gamma}$ as the Lagrange interpolation points to define the surface $\Gamma$, then the $O(h^\ell)$ term in \eqref{theorem_bound} is the error due to tensor-product $k+1$-point Gauss-Lobatto quadrature, which is exact for polynomials of order $2k-1$. Thus $\ell = 2k$ and $|\lambda_j-\Lambda_j|\lesssim h^{2r} + h^{2k}$. We demonstrate this numerically below. 
\end{remark}
\begin{remark} It follows from \eqref{mass-H} that computation of ${\rm area}(\gamma)$ using quadrature may also be superconvergent.  This has been observed numerically when using deal.ii \cite[Step 10 Tutorial]{BHK:07}.  
\end{remark}

\section{Numerical results for eigenvalue superconvergence} \label{sec7}
In this section we numerically investigate the convergence rate of the geometric term in the eigenvalue estimate of Theorem \ref{Superconvergence}. Using the upper bound we derived as a guide, we set the order $r$ of the PDE approximation so that $h^{2r}$ is higher order in the experiments.  

We first approximated the unit circle using a sequence of polygons with uniform faces.  For higher order approximations we interpolated the circle using equally spaced points and points based on Gauss-Lobatto quadrature. The left plot in Figure \ref{figcirceigs} shows convergence rates for $\lambda_1$ for various choices of $k$ for both spacings. The error when using Gauss-Lobatto points follows a trend of $h^{2k}$ as predicted by our analysis in Section \ref{sec6}. The errors when using equally spaced Lagrange points are  $O(h^{k+1})$ for odd values of $k$ and $O(h^{k+2})$ for even values of $k$. These quadrature errors arise from the Newton-Cotes rule corresponding to standard Lagrange points, yielding for example Simpson's rule with error $O(h^4)=O(h^{k+2})$ when $k=2$.   
 \setlength{\unitlength}{.75cm}
 \begin{figure}[h]
  	\label{figcirceigs}
 	\centering
 	\includegraphics[scale=.29]{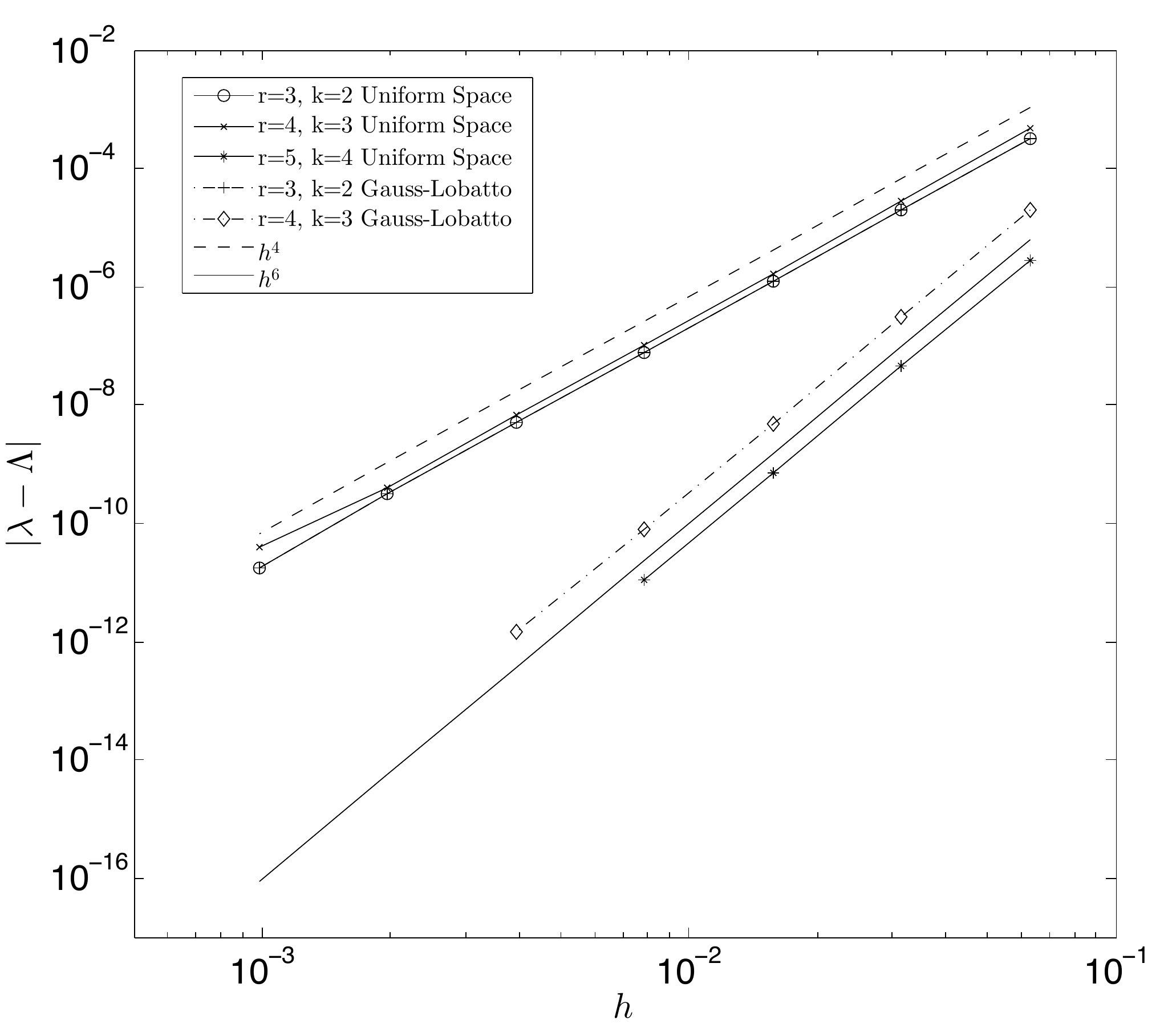}
		\includegraphics[scale=.33]{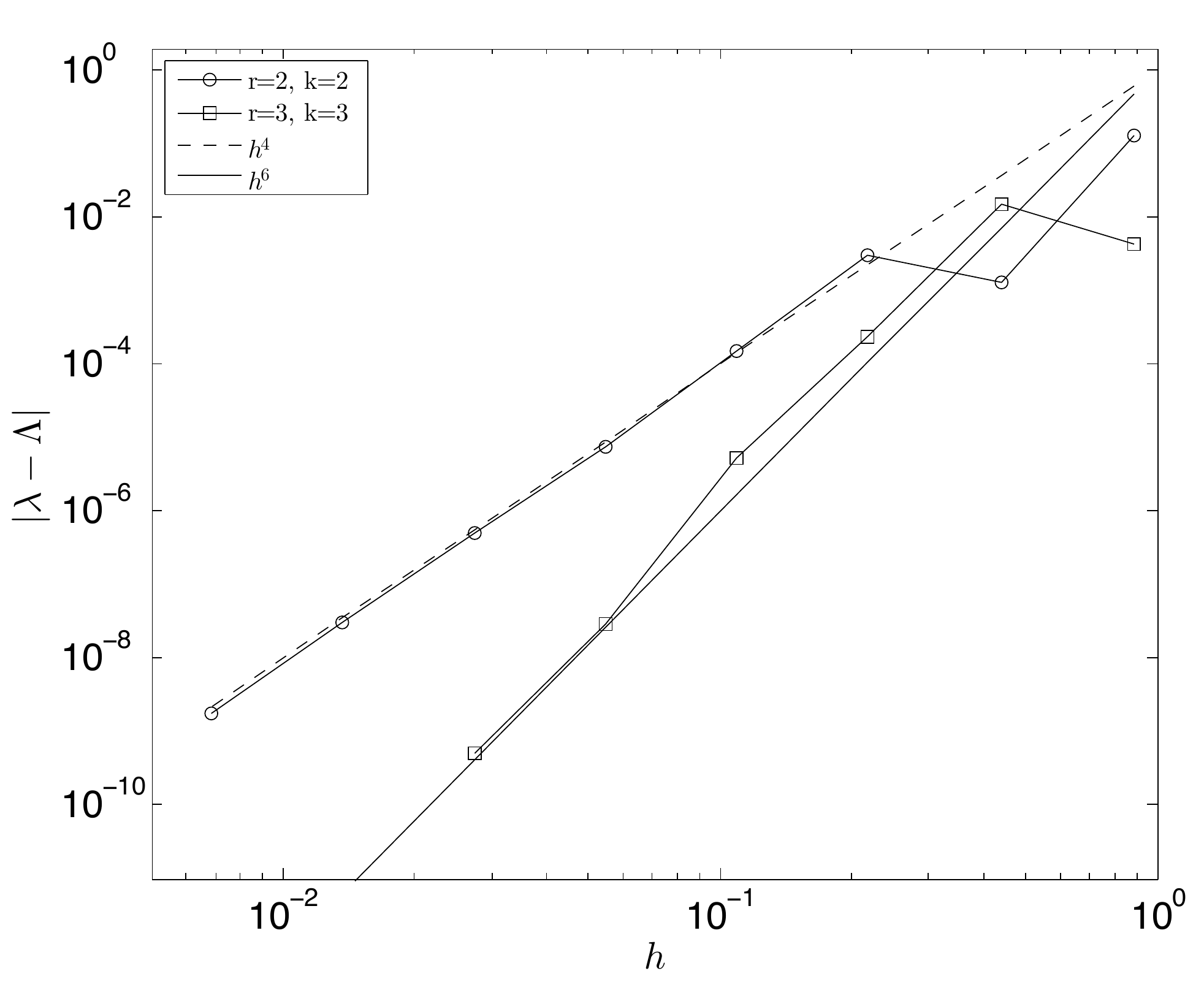}
 	\caption{{\it Left:}  Convergence rates of the first eigenvalue for the circle using typical equally spaced Lagrange basis points and Gauss-Lobatto Lagrange basis points.  {\it Right:} Convergence rates of the first eigenvalue for $(x-z^2)^2 + y^2 + z^2 +\frac{1}{2}(x-0.1)(y+0.1)(z+0.2) - 1 = 0$ surface using a quadrilateral mesh with Gauss-Lobatto Lagrange basis points.}
 \end{figure}
 
In our next experiment we used a quadrilateral mesh to approximate the surface $(x-z^2)^2 + y^2 + z^2 +\frac{1}{2}(x-0.1)(y+0.1)(z+0.2) - 1 = 0$. We used Gauss-Lobatto quadrature points on each face to construct the interpolated surface. Convergence rates for the first eigenvalue using $k=2,3$ are seen in the right plot in Figure \ref{figcirceigs}. The trend of order $h^{2k}$ convergence predicted by our analysis holds for surfaces in 2D when using Gauss-Lobatto interpolation points. Experiments yielding similar convergence rates were also performed on the sphere and torus.

We next investigated convergence on triangular meshes. We first created a triangulated approximation of the level set $(x-z^2)^2 + y^2 + z^2 - 1 = 0$ using standard Lagrange basis points. These points do not correspond to a known higher order quadrature rule.  In the left plot in Figure \ref{fighearteigs}, we see convergence rates of order $h^{k+1}$ for odd values of $k$ and $h^{k+2}$ for even values of $k$.  Unlike in one space dimension, these results cannot be directly proved using our framework above.  More subtle superconvergence phenomenon may provide an explanation.  For example, it is easy to show that the Newton-Cotes rule for $k=2$ corresponding to standard Lagrange interpolation points exactly integrates cubic polynomials on any two triangles forming a parallelogram.  It has previously been observed that meshes in which most triangle pairs form approximate parallelograms may lead to superconvergence effects, and it has been argued that many practical meshes fit within this framework; cf. \cite{XZ04}.  
\begin{figure}[h] 	\label{fighearteigs}
	\centering
	\includegraphics[scale=.41]{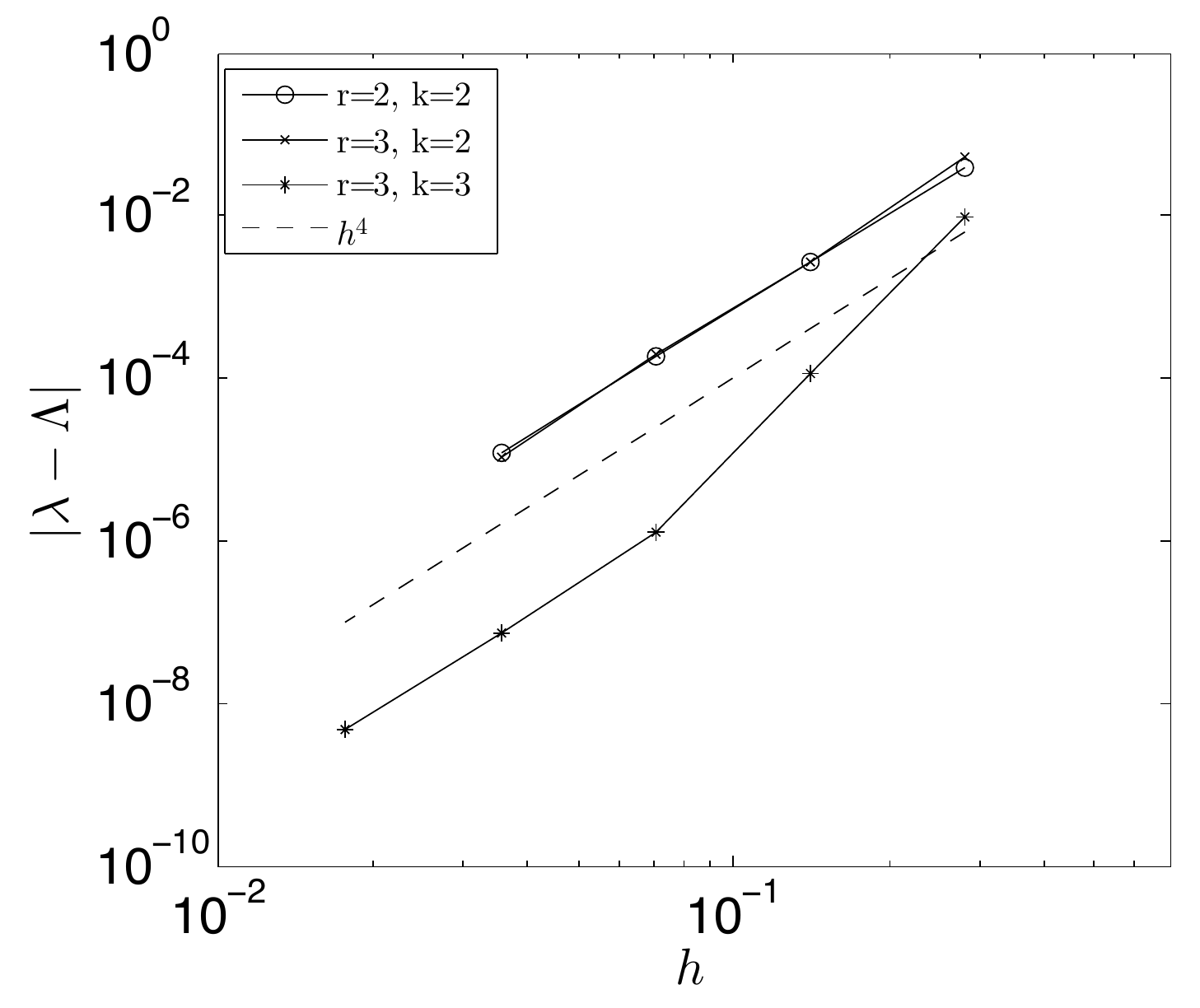}
		\includegraphics[scale=.47]{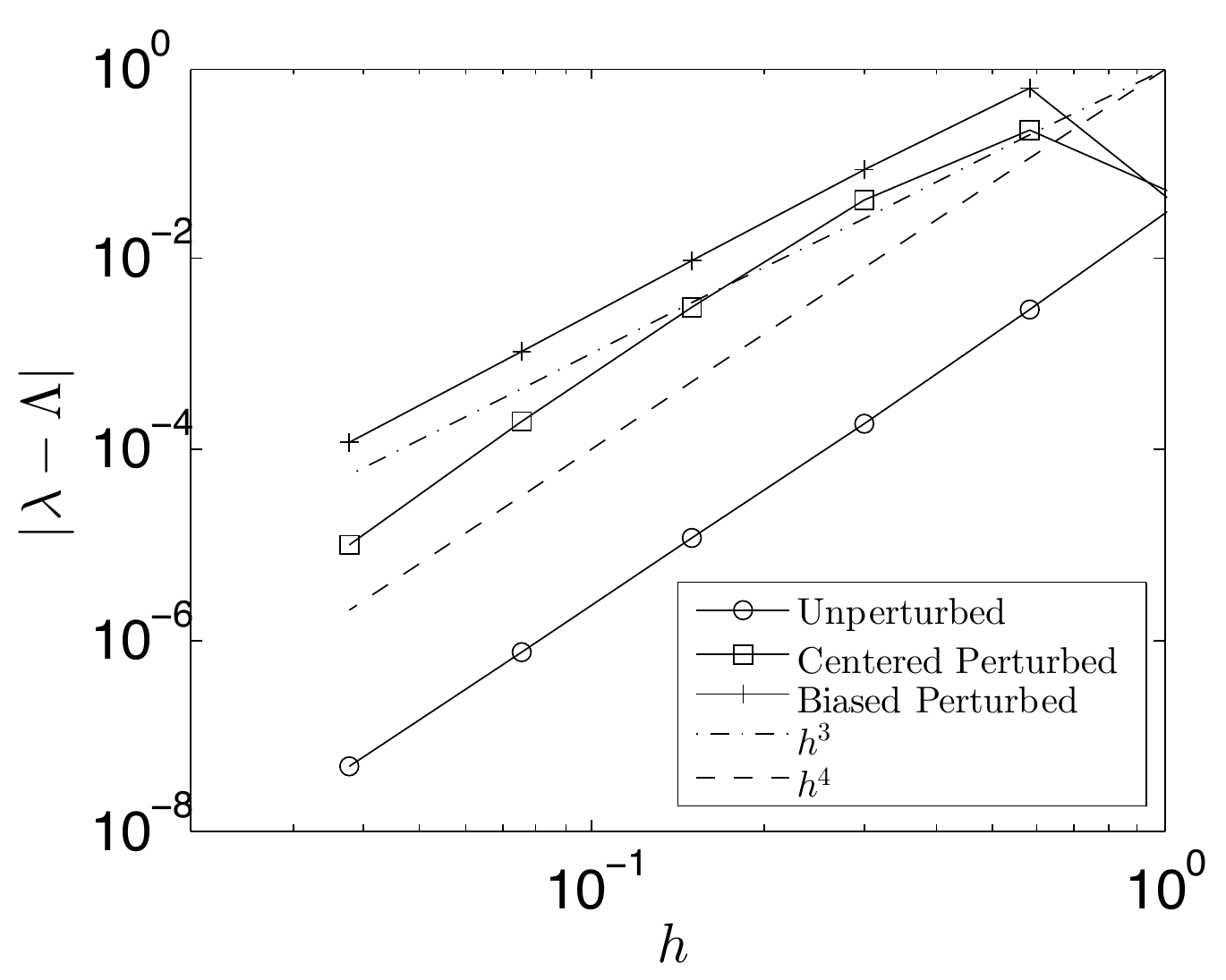}
	\caption{{\it Left:}. Convergence rates of an eigenvalue for $(x-z^2)^2 + y^2 + z^2 - 1 = 0$ surface using triangular mesh and typical Lagrange basis points.  {\it Right:} Convergence rates of the first eigenvalue for spherical surface using triangular mesh and unperturbed interpolation points, randomly perturbed interpolation points from a uniform distribution centered at 0 displacement, and randomly perturbed interpolation points from a uniform distribution centered at $0.5h^{k+1}$ displacement.}
\end{figure}

Finally, we attempted to break this even-odd superconvergence behavior by perturbing the points used to interpolate the sphere. First we perturbed points by $O(h^{k+1})$ using a uniform distribution on $h^{k+1}(-1,1)$. In expectation we then have a radial perturbation of 0. The superconvergence of $O(h^{k+2})$ for even $k$ values persisted for this situation. We then biased the previous distribution to be $h^{k+1}(-0.5,1.5)$ so that perturbations tended to be outward of the surface of the sphere. This led to convergence of $O(h^{k+1})$ for both even and odd values of $k$.  Numerical results for the error of the first eigenvalue of the sphere when $r=3$ and $k=2$ for an unperturbed sphere as well as these two perturbations are seen in the right plot in Figure \ref{fighearteigs}.

\begin{remark}
The perturbations of interpolation points on the sphere described above satisfy the abstract assumptions \eqref{e:d_estim} through \eqref{e:lift_estim} and so fit within the basic eigenvalue convergence theory of Section \ref{sec3}.  That theory is thus sharp without additional assumptions, but clearly does not satisfactorily explain many cases of interest.  
\end{remark}
\begin{remark}
The superconvergence effects we have observed appear to be relatively robust.  They may still occur even in applications where the continuous surface is not interpolated exactly as long as surface approximation errors at the interpolation points are uniformly distributed inside and outside of $\gamma$ with zero mean.  
\end{remark}

\bibliographystyle{siam}
\bibliography{Bibliography}

\end{document}